\documentclass[11pt,reqno]{amsart}%
\usepackage[margin=1in,letterpaper]{geometry}
\usepackage{graphicx}
\usepackage{amssymb}
\usepackage{amsthm}
\usepackage{mathrsfs}
\usepackage{stmaryrd}
\usepackage{accents} 
\usepackage{enumitem} 
\usepackage{subcaption}
\usepackage{microtype}
\usepackage{colonequals}
\usepackage{xcolor}

\usepackage[bookmarksopen,bookmarksdepth=2]{hyperref} 
\allowdisplaybreaks

\makeatletter\let\over\@@over\makeatother

\numberwithin{equation}{section}
\theoremstyle{plain} 
\newtheorem{theorem}{Theorem}[section] 
 
\newtheorem{corollary}[theorem]{Corollary}
\newtheorem{lemma}[theorem]{Lemma} 
\theoremstyle{remark}
\newtheorem{remark}[theorem]{Remark}
\theoremstyle{definition}
\newtheorem{definition}[theorem]{Definition}
\newtheorem{example}[theorem]{Example}

\newcommand{\be}{\begin{equation}}
\newcommand{\ee}{\end{equation}}%
\newcommand{\bse}{\begin{subequations}}
\newcommand{\ese}{\end{subequations}}

\newcommand{\realpart}{\operatorname{Re}}
\newcommand{\imagpart}{\operatorname{Im}}

\newcommand{\range}{\operatorname{rng}}
\newcommand{\kernel}{\operatorname{ker}}

\newcommand{\cod}{\operatorname{codim}}

\newcommand{\id}{\operatorname{id}}

\newcommand{\R}{\mathbb{R}}
\newcommand{\C}{\mathbb{C}} 
\newcommand{\placeholder}{\,\cdot\,}

\newcommand{\abs}[2][]{#1\lvert #2 #1\rvert}

\newcommand{\loc}{{\mathrm{loc}} } 
\newcommand{\even}{{\mathrm{e}} }

\newcommand\A{\mathscr A}    
\newcommand\B{\mathscr B}    
\providecommand{\G}{}        
\renewcommand\G{\mathscr G}  
\newcommand\F{\mathscr F}    
\newcommand\Vrho{\mathcal{V}}

\newcommand\Wspace{\mathscr W}
\newcommand\Xspace{\mathscr X}
\newcommand\Yspace{\mathscr Y}

\newcommand\Pspace{\mathcal P}

\newcommand{\cm}{{\mathscr C}}  
       
\newcommand{\km}{{\mathscr K}}            

\newcommand\fluidD{\mathscr{D}}

\newcommand{\confD}{\mathcal{D}}

\newcommand{\fullparam}{\Lambda}
\newcommand{\param}{\lambda}
\newcommand{\proj}{P}

\newcommand{\zc}{z_{\mathrm{c}}}

\newcommand{\per}{\mathrm{per}}

\newcommand{\cOmega}{\boldsymbol{\Omega}}
\newcommand{\cgamma}{\boldsymbol{\gamma}}     
\newcommand{\boundaryvelocity}{\mathfrak{v}}

\newcommand{\vortexp}{{p}}

\newcommand{\rot}{{\mathrm{rot}}}
\newcommand{\hstate}{{\mathrm{H}}}
\newcommand{\vstate}{{\mathrm{V}}}

\newcommand{\rotmat}{R}

\begin{document}

\title[Imploding hollow vortices]{Finite-time self-similar implosion of hollow vortices}

\date{\today}

\author[R. M. Chen]{Robin Ming Chen}
\address{Department of Mathematics, University of Pittsburgh, Pittsburgh, PA 15260} 
\email{mingchen@pitt.edu}  

\author[S. Walsh]{Samuel Walsh}
\address{Department of Mathematics, University of Missouri, Columbia, MO 65211} 
\email{walshsa@missouri.edu} 

\author[M. H. Wheeler]{Miles H. Wheeler}
\address{Department of Mathematical Sciences, University of Bath, Bath BA2 7AY, United Kingdom}
\email{mw2319@bath.ac.uk}

\begin{abstract}

In this paper, we consider the finite-time blowup of hollow vortices.  These are solutions of the two-dimensional Euler equations for which the fluid domain is the complement of finitely many Jordan curves $\Gamma_1, \ldots, \Gamma_M$, and such that the flow is irrotational and incompressible, but with a nonzero circulation around each boundary component.  The region bounded by $\Gamma_k$ is a ``vortex core'', modeled as a bubble of ideal gas: the pressure is constant in space and inversely proportional to the area of the vortex.  This can be thought of as the isobaric approximation assuming isothermal flow.

Our results come in two parts.  There exist explicit families of purely circular rotating and imploding hollow vortices. Implosion means more precisely that the vortex core shrinks to the origin in finite time, while the absolute value of the pressure simultaneously diverges to infinity. We prove that for any $m \geq 2$, there exist near-circular $m$-fold symmetric rotating hollow vortices.  By contrast, for all $m \geq 2$, the purely circular imploding vortices are locally unique among all collapsing vortices with uniform velocity at infinity.

The second part concerns configurations of multiple hollow vortices.  The existence of configurations of point vortices that collapse into a common point in finite time is classical.  We prove that generically, these can be desingularized to yield families of hollow vortex configurations exhibiting self-similar finite-time implosion.  Specific examples of an imploding trio and quartet of hollow vortices are given.  
\end{abstract}

\maketitle

\setcounter{tocdepth}{1}
\tableofcontents

\section{Introduction}
\label{introduction section}

Since their introduction by Helmholtz~\cite{helmholtz1858uber} in 1858, point vortices have served as an extremely useful idealized model for the evolution of regions of highly-concentrated vorticity in two dimensions.  Formally, the vorticity is represented by a sum of Dirac $\delta$-measures; the positions of the atoms, called the \emph{vortex centers}, are governed by a finite-dimensional Hamiltonian system derived from the two-dimensional incompressible Euler equations. It was discovered by Gröbli \cite{grobli1877specielle} in 1877 that point vortex dynamics admits remarkable collapsing solutions: there are configurations consisting of three initially distinct vortices that coalesce at a single point in finite time.  Their motion is self-similar, with each vortex center following a logarithmic spiral so that the relative configuration is preserved while shrinking in scale. Gröbli's collapsing vortex triples were rediscovered independently over a century later \cite{aref1979motion,novikov1979vortex,kimura1987similarity}; further studies have led to a complete characterization of the three-vortex system, while configurations consisting of many more vortices have also been investigated \cite{hernandez2007collisions,aref2010self,krishnamurthy2018finite,kallyadan2022selfsimilar}.

A fundamental question in vortex dynamics is whether the singularity formation observed in collapsing point vortices is realizable within the full two-dimensional incompressible Euler equations. One natural approach is to ``desingularize" point vortices into solutions where vorticity is concentrated around the trajectories of the vortex centers. The objective of this paper is to perform such a construction beginning with a (generic) collapsing point vortex configuration. However, in sharp contrast to the finite-time singularity formation for the point vortex system, classical solutions of the two-dimensional Euler equations are globally defined in time~\cite{holder1933unbeschrankte,lichtenstein1925einige,wolibner1933theoreme}, as are weak solutions with Lipschitz data due to a seminal work of Yudovich~\cite{yudovich1963non}.
Therefore vortex patches --- by far the most commonly studied variety of localized vorticity Eulerian flows --- unfortunately will not suffice. In this setting, the vorticity is compactly supported and the velocity field is locally Lipschitz, and hence the solutions are likewise defined globally in time. 
Even outside the Yudovich class, a classical result of Delort \cite{delort1991existence} ensures the global existence of weak Euler solutions for vortex sheet data with a distinguished sign. Thus, the challenge is to identify an intermediate vorticity model that is more regular than point vortices while below the threshold for global-in-time well-posedness. 

Self-similar solutions with singular trajectories provide one avenue for singularity formation in Euler flows. Elling \cite{elling2013algebraic,elling2016self} (see also \cite{shao2023self,shao2025self}) constructed algebraic spiral solutions, which emerge from an initial vorticity concentration at the origin and describe the roll-up of fluid into an infinite spiraling pattern. Logarithmic spirals provide another example, particularly in the context of vortex sheets, where finite-time singularity can occur due to the blowup of kinetic energy \cite{cieslak2024spiral} or the divergence of vortex sheet strength \cite{jeong2023logarithmic}.

In this paper, we investigate a different model of localized vorticity. \emph{Hollow vortices} are regions of spatially uniform pressure surrounded by a vortex sheet within an ideal fluid. They can also be viewed as classical regularity solutions of the free boundary Euler equations with boundary pressure forcing that is time dependent but constant on each boundary component. Alternatively, as we explain below, the vortex cores can be interpreted as a second compressible phase under the isobaric approximation. Hollow vortices have been studied since the 19th century~\cite{pocklington1894configuration}, and there has been a great deal of recent interest in steady hollow vortex configurations of various types~\cite{baker1976structure,crowdy2011analytical,llewellyn2012structure,crowdy2013translating,traizet2015hollow,zannetti2016hollow,chen2023desingularization}. In this work, we give the first construction of self-similarly collapsing hollow vortices. Unlike algebraic spirals, which exhibit singular vorticity transport along prescribed trajectories, or logarithmic spiral vortex sheets, where singularities develop due to energy or sheet-strength blowup, hollow vortex collapse is characterized by the shrinking of the core region to a point, leading to a pressure singularity (implosion). In that sense, the hollow vortex model offers a regularization of the collapsing point vortex configurations for which the finite-time breakdown can be understood through fluid mechanical principles. 

\subsection{Governing equations}
\label{intro governing equation section}

Let us now state the model more precisely. The fluid motion is taken to be governed by the incompressible Euler equations in the bulk
\begin{subequations}
\label{lab Euler}
\begin{equation}
  \label{lab momentum + incompressibility}
  \left\{
    \begin{aligned}
      \partial_t \mathbf{u} + (\mathbf{u} \cdot \nabla) \mathbf{u} & = -\nabla p \\
      \nabla \cdot \mathbf{u} & = 0 
    \end{aligned}
  \right. \qquad \textrm{in } \fluidD(t),
\end{equation}
where $\fluidD(t) \subset \mathbb{R}^2$ is the (time-dependent) fluid domain, $\mathbf{u} = \mathbf{u}(t,\placeholder) \colon \overline{\fluidD(t)} \to \mathbb{R}^2$ is the velocity field, and $p = p(t,\placeholder) \colon \overline{\fluidD(t)} \to \mathbb{R}$ is the pressure. In the hollow vortex setting, we assume that 
\[
	\fluidD(t)^c = \overline{\mathscr{V}_1(t) \cup \cdots \cup \mathscr{V}_M(t)} \equalscolon \overline{\mathscr{V}(t)},
\]
with $\mathscr{V}_1(t), \ldots, \mathscr{V}_M(t)$ disjoint regions --- the \emph{vortex cores} --- such that the boundary of $\mathscr{V}_k(t)$ is the Jordan curve $\Gamma_k(t)$. 

The kinematic condition states that the normal velocity $\boundaryvelocity_k = \boundaryvelocity_k(t,\placeholder) \colon \Gamma_k(t) \to \mathbb{R}^2$ of the $k$-th vortex boundary coincides with the normal velocity of the fluid:
\begin{equation}
  \label{lab kinematic}
  \mathbf{n} \cdot \mathbf{u} = \boundaryvelocity_k \qquad \textrm{on } \Gamma_k(t).
\end{equation}
The dynamic condition requires the pressure to be continuous over $\partial\fluidD(t)$.  For hollow vortices, each vortex core $\mathscr{V}_k(t)$ is taken to be a region of spatially constant pressure $\vortexp_k = \vortexp_k(t)$, and thus 
\begin{equation}
  \label{lab dynamic}
  p = \vortexp_k \qquad \textrm{on } \Gamma_k(t). 
\end{equation}
This condition on the pressure can be interpreted as an isobaric approximation, a common approximation in the study of gas bubbles where the internal pressure is assumed to be spatially uniform due to rapid sound-wave equilibration \cite{benjamin1987hamiltonian,prosperetti1991bubbles}. More broadly, \eqref{lab dynamic} can be viewed as a simplified model of bubble dynamics, similar to the framework studied in \cite{lai2024asymmetric}, where the gas pressure within a perturbed bubble remains uniform and evolves according to the ideal gas law. If we further assume an isothermal process, the product of pressure and volume remains constant. Consequently, in the setting of hollow vortices, the collapse of the vortex cores to a point in finite time forces their volumes to shrink to zero, driving the internal pressure to infinity in absolute value. This implosion process provides a natural mechanism for finite-time blowup of the system.

Finally, we require that the vorticity is concentrated on the vortex boundaries in the sense that the velocity field is irrotational in the interior of $\fluidD(t)$ but has a nonzero circulation $\gamma_k \in \mathbb{R}$ around the $k$-th vortex: 
\begin{equation}
\label{lab circulation}
  \begin{aligned}
    \nabla^\perp \cdot \mathbf{u} & = 0 &\qquad& \textrm{in } \fluidD(t) \\
    \int_{\Gamma_k(t)} \mathbf{t} \cdot \mathbf{u}  \, ds & = \gamma_k & \qquad & k = 1, \ldots, M, 
  \end{aligned}
\end{equation}
\end{subequations}
where $\mathbf{t}$ is the unit tangent vector and $\Gamma_k(t)$ is oriented counterclockwise.
Here, and elsewhere, we use a superscript $\perp$ to indicate the $90^\circ$ counterclockwise rotation of a two-dimensional vector field: $(x,y)^\perp \colonequals (-y,x)$ and similarly $\nabla^\perp = (\partial_x,\partial_y)^\perp = (-\partial_y,\partial_x)$.

Our objective in this paper is to construct solutions to~\eqref{lab Euler} that collapse self-similarly into the origin in finite time. With that in mind, we switch to a (non-inertial) reference frame with spatial variable $\boldsymbol \xi$ and time variable $T$ related to the coordinates $(\mathbf x, t)$ in the original reference frame by
\begin{equation}
  \label{definition T xi}
    t = \kappa(1-e^{-\frac T\kappa}),
  \qquad 
  \mathbf x = e^{-\frac T{2\kappa}}\rotmat(\Omega T) \boldsymbol \xi.
\end{equation}
Here, $\rotmat(\theta)$ denotes a counterclockwise rotation by the angle $\theta$ and $\kappa \in (0,+\infty]$ and $\Omega \in \R$ are constants. For $\kappa < \infty$ this coordinate system collapses as $t \nearrow \kappa$ (equivalently $T \to +\infty$), while for $\kappa = +\infty$ the first equation in \eqref{definition T xi} is interpreted as $T=t$ so that the new coordinates rotate at a constant angular velocity $\Omega$ without any contraction. 

One can check that the Eulerian velocity fields $\mathbf u = \mathbf u(\mathbf x,t)$ and $\mathbf U = \mathbf U(\boldsymbol \xi,T)$ in the two reference frames are related by
\begin{align}
  \label{relative velocity definition}
  \mathbf u(\mathbf x,t)
  = e^{-\frac T{2\kappa}}\rotmat(\Omega T) \mathbf U(\boldsymbol \xi,T)
  + e^{-\frac T\kappa}\bigg(\Omega {\mathbf x}^\perp - \frac 1{2\kappa} \mathbf x\bigg).
\end{align}
By this we mean that for trajectories $\mathbf x = \mathbf X(t)$ and $\boldsymbol \xi = \boldsymbol \Xi(T)$ related by \eqref{definition T xi}, the ordinary differential equation $\frac d{dt} \mathbf X = \mathbf u(\mathbf X,t)$ is equivalent to $\frac d{dT} \boldsymbol \Xi = \mathbf U(\boldsymbol \Xi,T)$. 
Inserting \eqref{relative velocity definition} into \eqref{lab momentum + incompressibility} we find
\begin{equation*}
  \label{spiraling Euler}
  \mathbf{U}_T + (\mathbf{U} \cdot \nabla) \mathbf{U} + 2\Omega \mathbf{U}^\perp - \left( \frac{1}{4\kappa^2} + \Omega^2 \right) \boldsymbol{\xi}    = -\nabla P  
\end{equation*}
for $\mathbf{x} \in \fluidD(t)$, where here $P$ is a rescaled pressure in the spiraling cordinates defined by
\begin{equation}
  \label{spiraling P}
  p(\mathbf x,t) = e^{-\frac T\kappa} P(\boldsymbol\xi,T).
\end{equation}

A \emph{self-similar} solution to the hollow vortex problem~\eqref{lab Euler} is one for which $\mathbf{U}$ and $P$ are independent of time $T$. The fluid domain then evolves via $\fluidD(t) = e^{-\frac T{2\kappa}} R(\Omega T)\fluidD(0)$. One can show that the \emph{steady problem} satisfied by $\mathbf U$ and $P$ is
\begin{subequations}
\label{steady hollow vortex problem}
\begin{equation}
  \label{steady hollow vortex problem interior}
  \left\{
    \begin{aligned} 
      \mathbf{U} \cdot \nabla  \mathbf{U} + 2\Omega \mathbf{U}^\perp - \left( \frac{1}{4\kappa^2} + \Omega^2 \right) \boldsymbol{\xi}   & = -\nabla P  \\
      \nabla \cdot  \mathbf{U}  & = \frac{1}{\kappa}  \\
      \nabla ^\perp \cdot  \mathbf{U}  & = -2\Omega  \\
    \end{aligned}
  \right.
  \qquad 
    \textrm{in } \fluidD
\end{equation}
together with the boundary conditions
\begin{equation}
  \label{steady hollow vortex problem boundary}
  \left\{
    \begin{aligned} 
      \int_{\Gamma_k} \mathbf{t} \cdot \mathbf{U}  \, ds & = \gamma_k + 2\Omega |\mathscr{V}_k| \\
      \mathbf{n} \cdot \mathbf{U} & = 0 & \qquad & \textrm{on } \Gamma_k \\
      P & = P_k & \qquad & \textrm{on } \Gamma_k
    \end{aligned}
  \right.
\end{equation}
\end{subequations}
for $k = 1, \ldots, M$, and where we are abbreviating $\fluidD \colonequals \fluidD(0)$, $\mathscr{V}_k \colonequals \mathscr{V}_k(0)$, and $\Gamma_k \colonequals \partial \mathscr{V}_k = \Gamma_k(0)$. Here, all functions and differential operators are in the $\boldsymbol\xi$ coordinates.  Each $P_k$ is constant (in time and space), so that the pressure of the $k$-th hollow vortex in the lab frame will be $p_k(t) = \exp(-T/\kappa)P_k$. This reflects the ideal gas law: the pressure in the core region is inversely proportional to the area $\abs{\mathscr{V}(t)}=\exp(-T/\kappa) \abs{\mathscr V(0)}$, diverging as $t \nearrow \kappa$.
Because it is set in a non-inertial frame, the momentum balance equation in~\eqref{steady hollow vortex problem interior} includes three fictitious forces: the familiar \emph{Coriolis force} in the direction  $\mathbf{U}^\perp$ and \emph{centrifugal force} in the radial direction $\boldsymbol{\xi}$ due to rotation, but also an \emph{Euler-type force} in the radial direction related to the acceleration of the rotation rate.  The additional (constant) contribution to the circulation in \eqref{steady hollow vortex problem boundary} arises when we take the line integral of the $\Omega \mathbf{x}^\perp$ term in~\eqref{relative velocity definition}. 

\subsection{Informal statement of results}
\label{intro results section}

We now explain the contributions of the present paper.  For the time being, this is done somewhat informally, as the precise versions are most conveniently expressed following a further reformulation of the problem discussed below.

Our first set of results concern the case of a single rotating or imploding hollow vortex centered at the origin.  For each $\gamma, \Omega \in \mathbb{R}$ and $\kappa \in (0,+\infty]$, there is a  solution of this type with the explicit form 
\begin{equation}
  \label{intro explicit circular solution}
  \mathscr{V} = B_1, \qquad \mathbf{U}_{\textrm{c}} = \mathbf{U}_{\textrm{c}}(\boldsymbol{\xi}; \, \gamma, \Omega, \kappa) \colonequals \frac{\gamma}{2\pi} \frac{\boldsymbol{\xi}^\perp}{|\boldsymbol{\xi}|^2} - \frac{1}{2\kappa} \frac{\boldsymbol{\xi}}{|\boldsymbol{\xi}|^2} + \frac{1}{2\kappa} \boldsymbol{\xi} - \Omega \boldsymbol{\xi}^\perp.
\end{equation}
As usual, in the case $\kappa = +\infty$, it is understood that all terms with coefficient $1/(2\kappa)$ vanish. Note that by a suitable rescaling, one can find such solutions with $\mathscr{V} = B_r$ for any $r > 0$. It is natural to then ask if there are nearby $m$-fold symmetric imploding and rotating or imploding hollow vortices.  Specifically, for each integer $m \geq 1$, a solution is said to be \emph{$m$-fold symmetric} provided the fluid domain $\fluidD$ and relative velocity $\mathbf{U}$ are invariant under the discrete symmetry group corresponding to rotations by $2\pi/m$ radians about the origin.

For the case of a purely rotating hollow vortex, there are in fact two families of near-circular $m$-fold solutions for any $m \geq 2$.

\begin{figure}
  \includegraphics[scale=1]{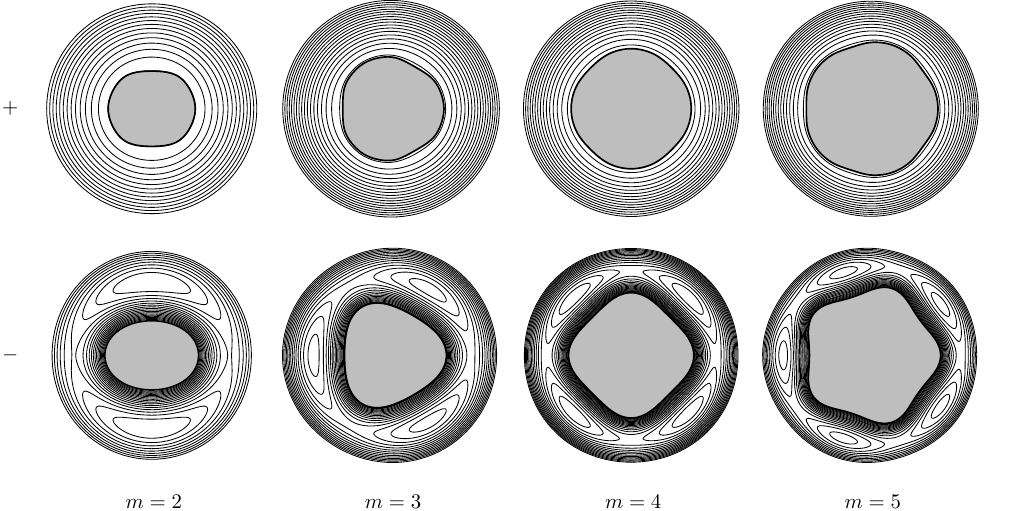}
  \caption{Sketched streamline patterns of rotating hollow vortices with $m$-fold symmetry for $m=2,3,4,5$; the first row correspond to solutions on $\cm_{\rot}^{m,+}$, while the second row are from $\cm_{\rot}^{m,-}$. The vortex is the shaded region, and the shape of the vortex and the streamlines are based on the leading-order approximation in \eqref{higher-order asymptotics}. Note that along the ``$+$'' branch, there are no critical layers, whereas on the ``$-$'' branch, one has the cat's eyes patterns typical of vortex patches and H-states.}
  \label{mfold figure}
\end{figure}

\begin{theorem}[Rotating hollow vortices]
\label{intro m-fold rotating theorem}
Fix a circulation $\gamma \neq 0$, integer $\ell \geq 0$, and $\alpha \in (0,1)$.  For any $m \geq 2$, take
\begin{equation}
  \label{definition Omega_0^pm}
  \Omega_{0,\pm} \colonequals \frac{\gamma}{2\pi} \left( 1  \pm \frac{1}{\sqrt{m}} \right).
\end{equation}
There exist real-analytic curves 
\[
	\cm_{\rot}^{m, \pm} = \left\{ (\mathbf{U}_\pm^\varepsilon, \, \mathscr{V}_\pm^\varepsilon, \Omega_\pm^\varepsilon ) : |\varepsilon| \ll 1 \right\}
\]
of $m$-fold symmetric rotating hollow vortex solutions bifurcating from the trivial solution~\eqref{intro explicit circular solution} in that 
\[
	\mathbf{U}_\pm^0 = \mathbf{U}_{\mathrm{c}}(\placeholder; \gamma, \Omega_{0,\pm},+\infty), \qquad \mathscr{V}_\pm^0 = B_1, \qquad \Omega_\pm^0 = \Omega_{0,\pm}.
\] 
  The vortex boundary $\Gamma_\pm^\varepsilon$ is the image under of the unit circle $\mathbb{T} \subset \mathbb{C}$ under a mapping $f_\pm^\varepsilon \in C^{\ell+1+\alpha}$ and is even about the $\xi$-axis. Moreover, we have the asymptotic expansions
\begin{equation}
  \label{R epsilon form}
  \begin{aligned}
    f_\pm^\varepsilon(e^{i\theta}) &= e^{i\theta}
    + \varepsilon e^{i(1-m)\theta}
    - \varepsilon^2 (\sqrt m \pm 1)^2 e^{i(1-2m)\theta}
    + O(\varepsilon^3)
     \quad \textrm{in } C^{\ell+1+\alpha},\\
     \Omega_\pm^\varepsilon &= 
     \Omega_{0,\pm}
\pm \varepsilon^2 \gamma\frac{(\sqrt m \pm 1)^3(\sqrt m \pm 3)}{4\pi\sqrt m} + O(\varepsilon^4).
  \end{aligned}
\end{equation}
\end{theorem}
The precise versions of this result is given in Theorems~\ref{m-fold rotating theorem}; formulas for the leading-order and form the of the relative velocity field are given in~\eqref{higher-order asymptotics}. Based on these, a sketch of the streamlines is provided in Figure~\ref{mfold figure}. 

Crowdy, Nelson, and Krishnamurthy~\cite{crowdy2021hstates} have previously studied rotating hollow vortices (that they call \emph{H-states}) of a related type and constructed families of exact solutions that limit to pinching. There is a subtle distinction, however, between their model and~\eqref{steady hollow vortex problem}. Integrating the steady Euler equations~\eqref{steady hollow vortex problem interior} gives Bernoulli's law in the self-similar variables
\begin{equation}
  \label{Bernoulli's law}
  \nabla \left( \frac{1}{2} |\mathbf{U}|^2 + P - \left( \frac{1}{4\kappa^2} + \Omega^2 \right) \frac{ |\boldsymbol{\xi}|^2}{2} \right) = 0 \qquad \textrm{in } \fluidD.
\end{equation}
Thus, for uniformly rotating solutions ($\kappa = \infty, \Omega \neq 0$), the total pressure $P_{\textrm{tot}} = P - \Omega^2 |\boldsymbol{\xi}|^2/2$ is a combination of $P$ and a centrifugal potential term. In~\cite{crowdy2021hstates}, the authors assume that $P_{\textrm{tot}}$ is constant on each boundary component, and hence the relative velocity field $\mathbf U$ likewise has constant magnitude there. The latter fact is quite useful for the complex-analytic methodology used in the paper, 
and this model has been used by other authors \cite{nelson2021corotating,chen2023desingularization}. However,
because we are interested in having as simple as possible a pressure condition in the lab frame, we instead ask in~\eqref{steady hollow vortex problem boundary} that the boundary values of $P$ (and hence $p$) are constant. Analytically, the result of this choice is that compared to~\cite{crowdy2021hstates} we have an additional inhomogeneous forcing term in the Bernoulli condition, which changes the dispersion relation and the qualitative form of the solutions. One immediately obvious difference, which can be seen in ~\eqref{definition Omega_0^pm}, is that there are two distinct  bifurcation curves at each $\gamma \neq 0$, whereas there is only one in the case of H-states.

We provide a comparison between the asymptotic form of our solutions and those of the H-states and rotating vortex patches (V-states) in Appendix~\ref{asymptotics appendix}. Qualitatively, the streamline pattern for the solutions along $\cm_{\rot}^{m,-}$ are similar to the H-states and V-states: both exhibit critical layers (where the angular fluid velocity $-\mathbf U \cdot \boldsymbol\xi^\perp/|\boldsymbol\xi|$ vanishes) and cat's eyes, though the radius at which the critical layer occurs is different in each case. Remarkably, though, the solutions on $\cm_{\rot}^{m,+}$ have no critical layers. The streamlines are instead \emph{monotone} in that they are globally polar graphs of functions that over each period decay monotonically from a local maximum (``crest'') to a local minimum (``trough'').  To the best of our knowledge, this structure is entirely different from any other previously constructed rotating vortex. Given the similarity between the hollow vortex problem and the steady gravity water wave problem, one might expect that $\cm_{\rot}^{m,+}$ can be globally continued up to a singular solution with an inward pointing $120^\circ$ corner at each trough. This is in contrast to the outward pointing $90^\circ$ corners observed in the case of V-states \cite{wu1984steady,overman1986steady,wang2024boundary}. It is less clear what one should expect along $\cm_{\rot}^{m,-}$, though it would be quite interesting to see if it exhibits the pinching observed in~\cite{crowdy2021hstates} for H-states.

When the collapse time is finite ($\kappa < \infty$), on the other hand, the purely circular imploding hollow vortices~\eqref{intro explicit circular solution} are actually locally unique.
For simplicity we state a version for the unit disk; analogous results for other disks can be obtained by a suitable rescaling.

\begin{theorem}[Rigidity] 
\label{intro rigidity corollary}
Fix $(\gamma^0, \Omega^0, \kappa^0)$ with $\kappa^0 < \infty$. There exists $\varepsilon > 0$ such that, if $\mathscr{V}$ is an $m$-fold symmetric hollow vortex with parameter values $(\gamma, \Omega, \kappa)$ and such that
\begin{enumerate}[label=\rm(\roman*\rm)]
	\item $|\mathscr{V}| = \pi$, 
	\item the relative velocity field vanishes at infinity, and
	\item $\partial \mathscr{V}$ is parameterized by $Z = Z(\theta) \in C_\per^{1+\alpha}([0,2\pi), \mathbb{C})$, with 
\[
	\| Z - e^{i\theta} \|_{C^{1+\alpha}} + |\gamma-\gamma^0| + |\kappa - \kappa^0| + |\Omega - \Omega^0| < \varepsilon,
\]
\end{enumerate}
then $\mathscr{V}$ is necessarily the unit disk and the velocity field is exactly $\mathbf{U}_{\textup{c}}(\gamma,\Omega,\kappa)$ as in~\eqref{intro explicit circular solution}.  
\end{theorem}

Our second set of results concern the existence of configurations of $M \geq 3$ hollow vortices that implode as they spiral self-similarly into the origin in finite time.  These are constructed through a vortex desinguarlization approach, starting from a collection of $M$ collapsing point vortices.  Recall that point vortex dynamics are a weakened version of the Euler system in which the vorticity is formally the sum of finitely many Dirac $\delta$ masses,
\begin{equation}
  \label{definition point vortex}
  \nabla^\perp \cdot \mathbf{u} = \sum_{k=1}^M \gamma_k \delta_{\boldsymbol{\xi}_k},
\end{equation}
where $\boldsymbol{\xi}_k \in \mathbb{R}^2$ is the center of the $k$-th vortex, and $\gamma_k \in \mathbb{R}$ is its circulation or vortex strength.  The well-known Kirchhoff--Helmholtz ODE governs the time evolution of the vortex centers.  Point vortex dynamics is a classical subject with a voluminous applied and pure literature; see, for example, \cite{saffman1992book,newton2001nvortex} and the references therein. While \eqref{definition point vortex} is not regular enough to constitute a weak solution to the Euler equations, the Kirchhoff--Helmholtz model can be obtained as the effective equation in the limit where the support of the vorticity shrinks to $\boldsymbol{\xi}_1,\ldots,\boldsymbol{\xi}_M$ in an appropriate sense.  The desingularization result we give in the present paper provides yet another justification of this type.

The Kirchhoff--Helmholtz model for a collapsing configuration reduces to a system of $2M$ (real) algebraic equations; written abstractly, they take the form $\mathcal{V}(\fullparam) = 0$, where the parameters are
\[
	\fullparam \colonequals (\boldsymbol{\xi}_1, \ldots, \boldsymbol{\xi}_M, \gamma_1, \ldots, \gamma_M, \kappa, \Omega) \in \mathbb{R}^{2M} \times \mathbb{R}^M \times \mathbb{R} \times \mathbb{R} \equalscolon \mathbb{P}.
\]
Briefly, this is derived analogously to the self-similar Euler equations, by asking that the vortex centers appear stationary in a rotating and contracting reference frame; see the discussion in Section~\ref{self-similar point vortex section}. We say that a point vortex configuration $\fullparam_0$ is \emph{non-degenerate} if the Jacobian of $\mathcal{V}$ has full rank at $\fullparam_0$.  Equivalently, there exists a decomposition of the parameter space $\mathbb{P} \cong \mathcal{P} \times \mathcal{P}^\prime$, for some $2M$-dimensional subspace $\mathcal{P}$ so that, writing $\fullparam = (\param,\param^\prime) \in \mathcal{P} \times \mathcal{P}^\prime$, we have that $D_\lambda\mathcal{V}(\fullparam_0)$ is an isomorphism; see Definition~\ref{definition non-degeneracy} and the surrounding discussion. 

\begin{figure}
  \includegraphics[scale=1]{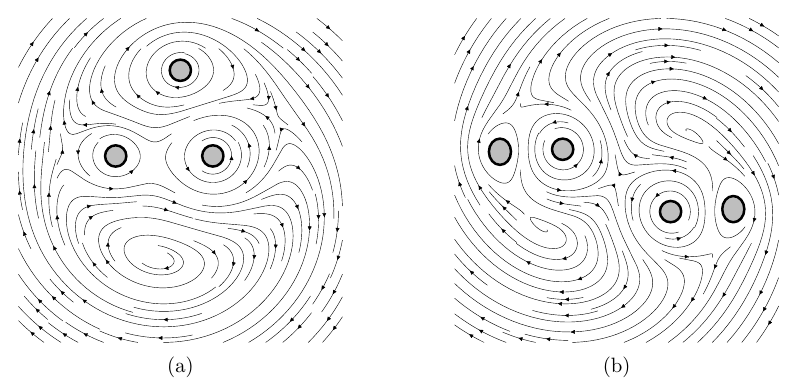}
  \caption{Sketched fluid particle trajectories, in the spiraling $(\boldsymbol\xi,T)$ frame \eqref{definition T xi}, for the collapsing hollow vortex trio (a) and quartet (b) in Corollary~\ref{triple quadruple corollary}. The velocity field $\mathbf{U}$, which has constant divergence $1/\kappa > 0$, is based on the corresponding point vortex configurations \eqref{point vortex trio} and \eqref{point vortex quartet}.}
  \label{collapse figure}
\end{figure}

Our main theorem states that any non-degenerate configuration of $M$ collapsing point vortices can be desingularized to create a family of $M$ self-similarly imploding hollow vortices.  

\begin{theorem}[$M$ imploding hollow vortices]
\label{intro configurations theorem}
Let $\fullparam_0 = (\param_0,\param_0^\prime)$ be a non-degenerate configuration of $M$ collapsing point vortices. There exists $\rho_1 > 0$, and a family $\km$ of imploding hollow vortex configurations that admits the real-analytic parameterization
\[
	\km = \left\{ \left( \mathbf{U}^\rho, \mathscr{V}^\rho, \lambda^\rho \right)   : \rho \in [0,\rho_1) \right\}
\]
and bifurcates from the point vortex configuration in the sense that
\[
	\mathbf{U}^0 = \sum_{k=1}^M \frac{\gamma_k}{2\pi } \frac{(\boldsymbol{\xi}-\boldsymbol{\xi}_k)^\perp}{|\boldsymbol{\xi}-\boldsymbol{\xi}_k|^2} + \frac{1}{2\kappa} \boldsymbol{\xi} - \Omega \boldsymbol{\xi}^\perp, \qquad \mathscr{V}^0 = \{ \boldsymbol{\xi}_1, \ldots, \boldsymbol{\xi}_M\}, \qquad \lambda^0 = \lambda_0.
\]
\end{theorem}

To our knowledge, this is the first rigorous construction of collapsing hollow vortices. Since the system is time-reversible, one also has the existence of expanding (and growing) hollow vortex configurations. We note that expanding arrangements of non-uniform vortex patches have recently been found by Dávila, del Pino, Musso, and Parmeshwar \cite{davila2024expanding}. In that paper the patches follow the trajectories of a self-similarly expanding point vortex configuration, but the full solution is not self-similar.
See Theorem~\ref{imploding hollow vortex configuration theorem} for 
the precise version of Theorem~\ref{intro configurations theorem} as well as leading-order asymptotics. The nondegeneracy assumption is standard for vortex desingularization results, and it can in principle be verified for any explicit configuration through elementary means.  In Corollary~\ref{triple quadruple corollary}, we provide two concrete examples, an imploding hollow vortex triple and quadruple, illustrated in Figure~\ref{collapse figure}.

\subsection{Plan of the article}
\label{intro plan section}

The steady fluid domain $\fluidD$ is unknown a priori, and so to facilitate the analysis we make a further change of variables in Section~\ref{preliminaries section}, recasting the problem on a fixed domain.   This comes at the usual expense of complicating the equations.  In particular, we take $\fluidD$ (complexified in the natural way) to be the image of a conformal mapping $f$ defined on the ``circular domain''
\[
	\confD_\rho \colonequals \mathbb{C} \setminus \overline{B_{|\rho|}(\zeta_1) \cup \cdots \cup B_{|\rho|}(\zeta_M)},
\]
for $\zeta_1, \ldots, \zeta_M \in \mathbb{C}$ and $0 < \abs \rho \ll 1$, so that $\Gamma_k = f(\partial B_\rho(\zeta_k))$.  The velocity field is then described via a complex potential $w$, which will be holomorphic but multi-valued on $\confD_\rho$.  Once their behavior at infinity is specified, $f$ and $w$ are determined by their traces on the $M$ components of the boundary.  Using layer-potential operators, they can be uniquely described in terms of real-valued densities $\mu_1, \ldots, \mu_M$ and $\nu_1, \ldots, \nu_M$, respectively, which become new unknowns.  Some important properties of layer-potential operators are reviewed in Section~\ref{layer-potential section}.  

Section~\ref{m-fold section} is devoted to the study of $m$-fold symmetric imploding and rotating hollow vortices.  As mentioned above, there is an explicit family of ``trivial solutions'' of this type~\eqref{intro explicit circular solution} consisting of a single ($M=1$) purely circular hollow vortex~\eqref{trivial solution}.  Through an explicit computation we show that the linearized operator at any of these circular vortices is Fredholm index $0$, and we obtain a dispersion relation~\eqref{definition dispersion relation} characterizing the parameter values at which the kernel is nontrivial. For any $m \geq 2$, the linearized operator in the imploding case is in fact always an isomorphism, which more-or-less gives the rigidity claimed in Theorem~\ref{intro rigidity corollary}.  For purely rotating vortices, there is a discrete set of parameter values at which the kernel is two dimensional. As the system enjoys reflection symmetries over the real axis, by working in the appropriate function spaces, we can then arrange for the kernel to be one dimensional so that the existence of near-circular $m$-fold symmetric follows from the classical Crandall--Rabinowitz theorem.  

In Section~\ref{multiple collapsing section}, we turn to the desingularization result Theorem~\ref{intro configurations theorem}.  The strategy is to once again reformulate the system using layer-potential operators, which leads to a nonlocal and nonlinear problem that can be written as the abstract operator equation $\F(u, \rho) = 0$.  The unknown $u$ bundles the densities $\mu_1,\ldots,\mu_M$ and $\nu_1, \ldots, \nu_M$, together with some of the parameters $\lambda$ describing the point vortex system. With some care, we show that this formulation captures both the point vortex system (when $\rho=0$) and the hollow vortex system (when $|\rho| > 0$). The main insight is that the linearization of the hollow vortex problem at a point vortex configuration is Fredholm, and the dimension of its kernel and cokernel match those of the linearized collapsing point vortex system.  At non-degenerate point vortex configurations, therefore, one may apply the implicit function theorem to infer the existence of nearby imploding hollow vortex configurations.  We then look at the two explicit examples in Figure~\ref{collapse figure}.

\section{Preliminaries}
\label{preliminaries section}

\subsection{Notation}
\label{notation section}

For complex variables $z\colonequals x+iy$ we will often make use of the Wirtinger derivative operator
\[
    \partial_z \colonequals \tfrac{1}{2} \left(  \partial_x - i \partial_y \right),
\]
which may also be denoted as $f_z = \partial_z f$ when there is no risk of confusion.
Primes are reserved for complex (total) derivatives of functions with domain $\mathbb{T}$, the unit circle in the complex plane; if $f$ is a function of $\tau = e^{i\theta} \in \mathbb{T}$, then $f' = \partial_\tau f = -ie^{-i\theta} df/d\theta$.  In particular, we caution that $\partial_\tau$ does not commute with $\realpart{(\placeholder)}$ or $\imagpart{(\placeholder)}$.

Let $D$ be a bounded subset of $\R^n$ or $\C^n$. For each integer $\ell \geq 0$ and $\alpha \in (0,1)$, we denote by $C^{\ell+\alpha}(D)$ the usual space of real-valued Hölder continuous functions of order $\ell$, exponent $\alpha$, and having domain $D$, with the usual norm.
Every $\varphi \in C^{\ell+\alpha}(\mathbb{T})$ admits a unique power series representation
\[
    \varphi(\tau) = \sum_{m \in \mathbb{Z}} \widehat{\varphi}_m \tau^m, \qquad \text{with}\qquad \widehat{\varphi}_m \colonequals \frac{1}{2\pi} \int_{\mathbb{T}} \varphi(\tau) \tau^{-m} \,d\theta.
\]
Note that because these are real-valued functions, the coefficients must necessarily obey $\widehat{\varphi}_{-m} = \overline{\widehat{\varphi}_m}$. For $m \geq 0$, let $\proj_m \colon C^{\ell+\alpha}(\mathbb{T}) \to C^{\ell+\alpha}(\mathbb{T})$ denote the projection
\[
    (\proj_m\varphi)(\tau) \colonequals
    \begin{cases}
        \widehat{\varphi}_m \tau^m + \widehat{\varphi}_{-m} \tau^{-m} & \text{if } m \neq 0 \\
        \widehat{\varphi}_0                                      & \text{if } m = 0,
    \end{cases}
\]
and set $\proj_{\leq m} \colonequals \proj_0 + \cdots + \proj_m$ and $\proj_{> m} \colonequals 1-\proj_{\leq m}$. We will often work with the space $\mathring{C}^{\ell+\alpha}(\mathbb{T}) \colonequals \proj_{> 0} C^{\ell+\alpha}(\mathbb{T})$ of mean $0$ elements of $C^{\ell+\alpha}(\mathbb{T})$.

\subsection{Complex-variable formulation}
\label{complex variable formulation section}

\subsubsection*{Bernoulli's law and velocity potentials}

Thanks to Bernoulli's law~\eqref{Bernoulli's law}, the dynamic boundary condition $P=P_k$ on $\Gamma_k$ in \eqref{steady hollow vortex problem boundary} can be rewritten as
\begin{equation}
  \label{steady dynamic condition}
   \frac{1}{2} |\mathbf{U}|^2 - \left( \frac{1}{4\kappa^2} + \Omega^2 \right) \frac{ |\boldsymbol{\xi}|^2}{2} 
= q_k \qquad \textrm{on } \Gamma_k,
\end{equation}
for some Bernoulli constants $q_1, \ldots, q_M \in \mathbb{R}$. Also notice that, by \eqref{steady hollow vortex problem interior}, the vector field
\begin{equation}
  \label{definition vector field W}
  \mathbf{U} -\frac{1}{2\kappa} \boldsymbol{\xi}+ \Omega \boldsymbol{\xi}^\perp
\end{equation}
is both curl free and divergence free. Thus there exist (possibly multivalued) potentials $\Psi,\Phi$ with 
\begin{equation}
  \mathbf{U} -\frac{1}{2\kappa} \boldsymbol{\xi}+ \Omega \boldsymbol{\xi}^\perp = \nabla\Phi = \nabla^\perp \Psi,
\end{equation}
which can be combined into a holomorphic potential
\begin{equation*}
  W \colonequals \Phi + i\Psi.
\end{equation*}

\subsubsection*{Conformal variables}
As discussed in Section~\ref{introduction section}, our approach will be to fix the steady fluid domain $\fluidD$ by viewing it as the image of a canonical domain $\confD_\rho$ under an a priori unknown conformal mapping $f$.  Specifically, for $M \geq 1$ vortices and conformal radius $\rho \in \mathbb{R}$, we take $\confD_\rho$ to be the circular domain 
\[
	\confD_\rho = \confD_\rho(\zeta_1, \ldots, \zeta_M)  \colonequals \mathbb{C} \setminus \overline{ B_{|\rho|}(\zeta_1) \cup \cdots \cup B_{|\rho|}(\zeta_M)}.
\]
Here $\zeta_1, \ldots, \zeta_M \in \mathbb{C}$ are distinct and the closures of the disks are disjoint.  In the vortex desingularization argument leading to Theorem~\ref{intro configurations theorem}, they limit to the vortex centers as $\rho \to 0$; for the $m$-fold symmetric rotating and imploding vortices, we simply take $\zeta_1 = 0$ and $\rho =1$.  The fluid domain $\fluidD$ is then defined as $f(\confD_\rho)$, and $\Gamma_k = f(\partial B_\rho(\zeta_k))$ is the the boundary of the $k$-th vortex.  Note that for this representation to be valid, we must have that $f$ is conformal on $\confD_\rho$, with a $C^1$ extension to $\overline{\confD_\rho}$ that is injective.  As a further normalization, we require that $f - \id$ vanishes at infinity. By convention, we think of $\confD_\rho$ as lying in the complex $\zeta$-plane, while the physical domain $\fluidD$ lies in the $z$-plane.   

Defining $w = W \circ f$ on $\confD_\rho$, which like $W$ will be holomorphic but multi-valued,
the relative fluid velocity $\mathbf U \equalscolon (U,V)$ can be expressed as
\begin{equation}
\label{definition potential w}
  U-iV = \frac{w_\zeta}{f_\zeta} + i\cOmega \bar f,
\end{equation}
where we have introduced the shorthand 
\begin{align}
  \label{cOmega definition}
  i\cOmega \colonequals i\Omega + \frac 1{2\kappa}.
\end{align}
Conversely, any velocity field $\mathbf U$ of this form automatically solves \eqref{steady hollow vortex problem interior} with the pressure $P$ recovered (up to an additive constant) from \eqref{Bernoulli's law}. The circulation and kinematic conditions in~\eqref{steady hollow vortex problem boundary} are now equivalent to
\begin{subequations}
\label{conformal governing equations}
\begin{align}
  \label{circulation condition}
  \realpart \int_{\partial B_\rho(\zeta_k)} w_\zeta \, d\zeta &= \gamma_k\\
  \intertext{and}
  \label{kinematic condition}
  \realpart{\left( \zeta f_\zeta \left( \frac{w_\zeta}{f_\zeta} + i\cOmega \bar f \right) \right)} &= 0 \qquad \textrm{on } \partial \confD_\rho,
\end{align}
respectively, while the dynamic condition~\eqref{steady dynamic condition} takes the form
\begin{equation}
  \label{dynamic condition}
  \frac{1}{2} \left| \frac{w_\zeta}{f_\zeta} + i\cOmega \bar f \right|^2 -\frac{1}{2}  |\cOmega|^2  |f|^2 = q_k \qquad \textrm{on } \partial \confD_\rho.
\end{equation}
\end{subequations}

\subsection{Layer-potential operators}
\label{layer-potential section}

As the potential $w$ and conformal mapping $f$ are holomorphic, it suffices to determine their boundary traces and the circulations $\gamma_1, \ldots, \gamma_M$.  When we carry out the vortex desingularization argument in Section~\ref{multiple collapsing section}, however, it will be necessary to have a framework that permits sending the conformal radius $\rho$ to $0$.  We therefore work with unknowns having domain $\mathbb{T}^M$ rather than $\partial\confD_\rho$.   The approach taken here follows very closely that of \cite{chen2023desingularization}, where translating, rotating, and stationary hollow vortex configurations were considered.

With that in mind, for $(\rho, \zeta_1, \ldots, \zeta_M)$ and real-valued densities $ \mu = (\mu_1, \ldots, \mu_M) \in \mathring{C}^{\ell+\alpha}(\mathbb{T})^M$, we define the layer-potential operator $\mathcal{Z}^\rho(\zeta_1, \ldots, \zeta_M)$ by
\begin{equation}
  \label{definition Z operator}
   \mathcal{Z}^\rho(\zeta_1, \ldots, \zeta_M)[\mu](\zeta) \colonequals  \frac{1}{2 \pi i} \sum_{k=1}^M  \int_{\mathbb{T}} \frac{\mu_k(\sigma)}{\rho \sigma + \zeta_k - \zeta} \, \rho d\sigma.
\end{equation}
To avoid cluttered notation, we will usually suppress the dependence on $\zeta_1, \ldots, \zeta_M$ when there is no risk of confusion.  It is clear from above that $\mathcal{Z}^\rho \mu$ gives a single-valued holomorphic function in $\confD_\rho$ that vanishes at infinity.  Moreover, it is continuous up to $\partial\confD_\rho$, and its trace can be evaluated using the Sokhotski--Plemelj formula.  For this, we introduce the operators
\begin{equation}
 \label{trace Z operator}
  \begin{aligned}
    \mathcal{Z}_k^\rho[\mu](\tau) & \colonequals  \mathcal{Z}^\rho[\mu](\zeta_k + \rho \tau) \\ 
    & = \frac{1}{2\pi i} \int_{\mathbb{T}} \frac{\mu_k(\sigma) - \mu_k(\tau)}{\sigma - \tau} \, d\sigma + \sum_{j \neq k} \frac{1}{2\pi i} \int_{\mathbb{T}} \frac{\mu_j(\sigma)}{\rho \sigma + \zeta_j - \rho \tau - \zeta_k} \rho \, d\sigma.
  \end{aligned}
\end{equation}
Note that by Privalov's theorem, $\mathcal{Z}^\rho_k$ is bounded $C^{\ell+\alpha}(\mathbb{T})^M \to C^{\ell+\alpha}(\mathbb{T},\mathbb C)$, for any $\ell \geq 0$ and $\alpha \in (0,1)$.  Moreover, $\mathcal{Z}_k^0 \mu = \mathcal{C} \mu_k$, where $\mathcal{C}$ is the Cauchy-type integral operator
\begin{equation}
  \label{definition C operator}
  \mathcal{C}[\mu_k](\tau) \colonequals  \frac{1}{2 \pi i} \int_{\mathbb{T}} \frac{\mu_k(\sigma) - \mu_k(\tau)}{\sigma - \tau} \, d\sigma.
\end{equation}
One can verify that $\mathcal{C} \tau^m = 0$ for $m \ge 0$ and $\mathcal{C} \tau^m = -\tau^m$ for $m < 0$.  Of course, if $M = 1$, then $\mathcal{Z}_k^\rho = \mathcal{C}$ as well.

\subsection{Abstract local bifurcation theorems}

For the convenience of the reader, we recall here the standard Crandall--Rabinowitz theorem on bifurcation from a simple eigenvalue.  This will be the main tool in the case of $m$-fold symmetric rotating hollow vortices. Note that the classical version of this theorem due to Crandall and Rabinowitz~\cite{crandall1971bifurcation} only assumes finite-regularity of the abstract operator. The proof of the real-analytic version quoted below can be found in~\cite[Theorem 8.3.1]{buffoni2003analytic}, for example.

\begin{theorem}[Crandall--Rabinowitz local bifurcation]
\label{CR theorem}
Let $\Xspace$ and $\Yspace$ be (real) Banach spaces and $\mathscr{O}$ an open subset of $\Xspace \times \mathbb{R}$.  Consider a real-analytic mapping 
\[
	\F = \F(u,\lambda) \colon \mathscr{O} \subset \Xspace \times \mathbb{R} \to \Yspace.
\]
Assume that 
\begin{enumerate}[label=\rm(\roman*\rm)]
	\item $\F(0,\lambda) = 0$ for all $(0,\lambda) \in \mathscr{O}$;
	\item there exists $(0,\lambda_0) \in \mathscr{O}$ such that $D_u \mathscr F(0,\lambda_0) \colon \Xspace \to \Yspace$ is Fredholm index zero with one-dimensional kernel generated by $u_0$; and
	\item \label{CR transversality} $D_\lambda D_u \F(0,\lambda_0) u_0 \not\in \range{D_u \F(0,\lambda_0)}$.
\end{enumerate}
There exists a curve $\cm_\loc \subset \F^{-1}(0)$ admitting the real-analytic parameterization
\[
	\cm_\loc = \left\{ (u^\varepsilon, \lambda^\varepsilon) : |\varepsilon| \ll 1 \right\} \subset \F^{-1}(0) 
\]
and such that
\[
	u^\varepsilon = \varepsilon u_0 + O(\varepsilon) \quad \textrm{in } \Xspace, \qquad \lambda^\varepsilon = \lambda_0 + O(\varepsilon).
\]
Locally, $\cm_\loc$ comprises all elements of the zero-set of $\F$ aside from the family of trivial solutions of the form $(0,\lambda)$.
\end{theorem}

\section{Rotating hollow vortices with \texorpdfstring{$m$}{m}-fold symmetry} 
\label{m-fold section}

For any $\cOmega \in \mathbb{C}$ and $\gamma \in \mathbb{R}$, there is an explicit trivial solution to the self-similar hollow vortex problem \eqref{conformal governing equations} consisting of a single purely circular vortex $\mathscr{V} = B_1$ with complex circulation 
\begin{equation}
 \label{circular vortex gamma}
 \cgamma = \cgamma^0(\gamma, \cOmega) \colonequals \gamma+2\pi i \imagpart{\cOmega} = \gamma + \frac{\pi}{\kappa} i,
\end{equation}
and having the conformal description
\begin{equation}
  \label{trivial solution}
  f^0 \colonequals \id, \qquad w^0 = w^0(\zeta;\cOmega) \colonequals \frac{\cgamma}{2\pi i} \log{\zeta}, \qquad q^0 = q^0(\cOmega) \colonequals \frac{1}{2} \left|\frac{\cgamma}{2\pi} - \cOmega\right|^2 - \frac{1}{2} |\cOmega|^2.
\end{equation}
These are simply the family of vortices~\eqref{intro explicit circular solution} written now in the complex-analytic formulation.  In the general framework of the previous section, this corresponds to taking $M = 1$ and $\rho = 1$.  For simplicity, in the analysis to follow, we take $\gamma \in \mathbb{R}$ to be fixed, so that $\cgamma$ is determined by $\cOmega$.  Notice that the imaginary part of $\cgamma$ corresponds to a ``phantom sink'' at the origin.

The purpose of this section is to characterize the set of $m$-fold symmetric (nontrivial) collapsing or rotating hollow vortices in a neighborhood of the trivial solution~\eqref{trivial solution}.    Recall that, in the conformal variables, $m$-fold symmetry refers to the invariance of the system under rotations by $2\pi/m$ in the following sense:
\begin{equation}
  \label{m-fold symmetry}
   f(e^{2\pi i/m} \placeholder) = e^{2\pi i/m}  f, \qquad w_\zeta( e^{2\pi i/m} \placeholder) = e^{-2\pi i/m}  w_\zeta.
\end{equation}
Of course, $m$-fold symmetry implies $\ell$-fold symmetry for any divisor $\ell \geq 2$ of $m$. In many other settings, one can further require that the vortex configurations exhibits the reflection symmetry
\begin{equation}
  \label{even symmetry}
  f(\overline{\zeta}) = \overline{f(\zeta)}, \end{equation}
which has the effect of removing translation invariance in the vertical direction.   For example, this is done in the construction of rotating $m$-fold symmetric vortex patches in \cite{hassainia2020global,garcia2023global} and cotranslating hollow vortex pairs in \cite{chen2023desingularization}.  On the other hand, the equations for collapsing hollow vortices do not respect this symmetry, and so one does not expect to find solutions satisfying \eqref{even symmetry}.

\subsection{Functional analytic setting}

It will be convenient to impose an ansatz on both the conformal mapping complex potential
\begin{equation}
  \label{w and f ansatz}
  f(\zeta) \equalscolon \zeta + \sum_{n \geq 1} \frac{\hat f_n}{\zeta^{mn-1}}, \qquad  w(\zeta) \equalscolon \frac{\cgamma}{2 \pi i} \log{\zeta}  + \sum_{n \geq 1} \frac{\hat w_n}{\zeta^{mn}}.
\end{equation}
One can check that this implies $f$ and $w-w^0$ are holomorphic on $\confD$, obey the required behavior as $\zeta \to \infty$, have the circulation \eqref{circulation condition}, and exhibit the desired $m$-fold symmetry \eqref{m-fold symmetry}.  In fact, this representation is locally unique, at least under the assumption that the vortex core is conformally equivalent to the unit disk; see Lemmas~\ref{uniqueness lemma} and \ref{near identity lemma}.

As both the kinematic condition \eqref{kinematic condition} and dynamic condition \eqref{dynamic condition} are posed on the vortex boundary, it will be sufficient to consider the traces of $f-\id$ and $w-w^0$ on $\mathbb{T}$.  In keeping with our notation elsewhere, we introduce the densities $\mu$ and $\nu$ so that 
\begin{equation}
  \label{m-fold trace ansatz}
  (f-\id)\Big|_{\mathbb{T}} \equalscolon \mathcal{C} \mu, \qquad \left(w - \frac{\cgamma}{2\pi i} \log{(\placeholder)} \right)\Big|_{\mathbb{T}} \equalscolon \mathcal{C} \nu.
\end{equation}
The symmetry requirements \eqref{m-fold symmetry} and \eqref{even symmetry} are encoded into the functional analytic setting by seeking $(\mu,\nu) \in \Xspace$, for the Banach space $\Xspace = \Xspace_1 \times \Xspace_2$ given by
\begin{equation}
\label{definition X space}
  \begin{aligned}
    \Xspace_1 & \colonequals \Big\{ \mu \in \mathring{C}^{k+1+\alpha}(\mathbb{T}) : \mu(\tau) = 2 \realpart \sum_{n \geq 1} \frac{\hat \mu_n}{\tau^{mn-1}} \Big\} \\
    \Xspace_2 & \colonequals \Big\{ \nu \in \mathring{C}^{k+1+\alpha}(\mathbb{T}) : \nu(\tau) = 2 \realpart \sum_{n \geq 1} \frac{\hat \nu_n}{\tau^{mn}}  \Big\}.
  \end{aligned}
\end{equation}
Here $k \geq 0$ and $\alpha \in (0,1)$ are fixed but arbitrary.  (Note that the indexing convention above is somewhat different than before.  Moreover, since $\mathcal{C} \tau^{-mn} = -\tau^{-mn}$,  comparing this to \eqref{w and f ansatz}, we have $\hat \mu_n = -\hat f_n$ and $\hat \nu_n = -\hat w_n$.) 

Finally, for these to give physical solutions to the problem, we must have that $f = \id + \mathcal{C} \mu$ is conformal on $\mathcal{D}$ and injective on $\overline{\mathcal{D}}$.  By the Darboux--Picard theorem and the minimum modulus principle, it suffices to ensure injectivity holds for $f|_{\Gamma}$.  
For near circular hollow vortices, these requirements can therefore be described by membership in the open set
\[
	\mathscr{O} \colonequals 
		\left\{ 
			(\mu,\nu, q, \cOmega) \in \Xspace \times \mathbb{R} \times \mathbb{C} 
				\colon 
			\inf_{\mathbb{T}} | 1+\mathcal{C}\mu^\prime | > 0
		\right \}.
\]

We can now write the self-similar hollow vortex problem \eqref{conformal governing equations}  abstractly as 
\[
	\F(u;\cOmega) = 0,
\]
where $u \colonequals (\mu,\nu,q)$, $\F = (\F_1, \F_2) \colon \mathscr{O}  \to \Yspace$ is the real-analytic mapping 
\begin{equation}
\label{definition of F} 
  \begin{aligned}
    \F_1(u; \cOmega) & \colonequals \realpart{\left( \tau \left( \frac{\cgamma}{2\pi i \tau} + \mathcal{C} \nu^\prime + i \cOmega (1+ \mathcal{C}\mu^\prime) \overline{(\tau+\mathcal{C}\mu)} \right) \right)}  \\
    \F_2(u; \cOmega) & \colonequals \frac{1}{2} \frac{\left| \frac{\cgamma}{2\pi i \tau} + \mathcal{C} \nu^\prime + i \cOmega (1+\mathcal{C}\mu^\prime) \overline{(\tau+\mathcal{C}\mu)}\right|^2}{|1+\mathcal{C}\mu^\prime|^2} - \frac{1}{2} |\cOmega|^2|\tau+\mathcal{C}\mu|^2 -q,
  \end{aligned}
\end{equation}
and the codomain $\Yspace = \Yspace_1 \times \Yspace_2$ is given by
\[
	\Yspace_1  \colonequals \Big\{ \varphi \in \mathring{C}^{k+\alpha}(\mathbb{T}) : \varphi(\tau) = 2 \realpart{\sum_{n\geq 1} \frac{\hat \varphi_n}{\tau^{mn}} } \Big\}, \quad
	\Yspace_2  \colonequals \Big\{ \varphi \in {C}^{k+\alpha}(\mathbb{T}) : \varphi(\tau) = 2 \realpart{\sum_{n\geq 0} \frac{\hat \varphi_n}{\tau^{mn}} }  \Big\}.
\]
Note that $\F_1$ corresponds to the kinematic condition~\eqref{kinematic condition}, while $\F_2$ represents the dynamic condition~\eqref{dynamic condition}.  The trivial solution \eqref{trivial solution} corresponds to $(\mu,\nu,q) = (0,0,q^0(\cOmega))$.  Recall that we will continue to think of $\gamma \neq 0$ being fixed, and setting $\imagpart{\cgamma}/2\pi = \imagpart{\cOmega}$.

For a (non-collapsing) rotating hollow vortex, we have $\cOmega = \Omega$ and $\cgamma = \gamma$, so that the nonlinear operator $\F$ exhibits certain reflection symmetries.  Taking advantage of these will allow us to reduce the dimension of the kernel and cokernel.  Defining the subspaces 
\begin{equation}
  \label{definition reflection X Y spaces}
  \begin{aligned}
    \Xspace_\even & \colonequals \left\{ (\mu,\nu) \in \Xspace \colon  \hat\mu_n \in \mathbb{R},~\hat \nu_n \in i\mathbb{R} \textrm{ for all } n \geq 1 \right\} \\
    \Yspace_\even & \colonequals \left\{ (\varphi, \psi) \in \Yspace \colon \hat\varphi_n \in i\mathbb{R},~\hat \psi_n \in \mathbb{R} \textrm{ for all } n \geq 1 \right\},
  \end{aligned}
\end{equation}
one can readily verify that the restriction $\F \colon \mathscr{O}_\even  \to \Yspace_\even$ is well-defined and real-analytic, for the restricted neighborhood $\mathscr{O}_\even \colonequals \mathscr{O} \cap (\Xspace_\even \times \mathbb{R} \times \mathbb{R})$.

\subsection{Linear analysis}

We now seek to characterize the kernel and range of the linearized operator 
\begin{equation}
  \label{definition L}
  \mathscr{L} \colonequals  D_{(\mu,\nu,q)} \F(0,0,q^0; \cOmega^0) \colon \Xspace \times \mathbb{R}  \to \Yspace
\end{equation}
at a trivial solution.  An easy calculation shows that
\begin{equation}
  \label{linearized kinematic condition}
  D_{(\mu,\nu,q)}\F_{1}(0,0,q^0; \cOmega) \begin{pmatrix} 
    \dot \mu \\ 
    \dot \nu \\ 
    \dot q 
  \end{pmatrix}
    = \realpart{\left( \tau \mathcal{C} \dot \nu^\prime + i \cOmega \left( \mathcal{C}\dot\mu^\prime + \tau \overline{\mathcal{C}\dot\mu}  \right) \right)  }.	
\end{equation}
Similarly, the linearized Bernoulli condition is 
\begin{equation}
  \label{linearized Bernoulli condition}
  \begin{aligned}
    D_{(\mu,\nu,q)}\F_{2}(0,0,q^0; \cOmega) \begin{pmatrix} 
      \dot \mu \\ 
      \dot \nu \\ 
      \dot q 
    \end{pmatrix}
    & = -\realpart{ \left(  \left( \frac{\overline{\cgamma}}{2\pi i }  + i \overline{\cOmega}  \right)   \left( \tau \mathcal{C} \dot\nu^\prime -\frac{\cgamma}{2\pi i } \mathcal{C}\dot \mu^\prime + i \cOmega \tau \overline{\mathcal{C}\dot \mu} \right) \right)} \\
    & \qquad -  |\cOmega|^2 \realpart{\left( \bar\tau \mathcal{C} \dot \mu \right)}  - \dot q.
  \end{aligned}
\end{equation}

We begin by identifying the null space of the linearized problem at the trivial solution.  Note that $\Xspace$ is treated as a real vector spaces when counting dimensions.

\begin{lemma}[Null space] 
\label{null space lemma}
For a fixed $m \geq 2$ and $\cOmega \in \mathbb{C}$, consider the null space $\kernel{\mathscr{L}}$ of the linearized operator $\mathscr{L}$ given by \eqref{definition L}.
\begin{enumerate}[label=\rm(\alph*\rm)]
	\item \label{dispersion relation part} \textup{(Dispersion relation)} The dimension of $\kernel{\mathscr{L}}$ is twice the number of integers $n \geq 1$ satisfying the dispersion relation
    \begin{equation}
      \label{definition dispersion relation}
      0 = d_m(n,\cOmega) \colonequals (1-mn) \left| \Omega - \frac{\gamma}{2\pi} \right|^2 + \frac{\gamma}{\pi} \overline{\cOmega} - \overline{\cOmega}^2.
    \end{equation}
	For any $\cOmega \neq 0$, there is at most one $n \geq 1$ satisfying \eqref{definition dispersion relation}.  If $n$ is a root, then $\{ \dot v_n^1, \, \dot v_n^2 \}$ form a basis (over $\mathbb{R})$ for $\kernel{\mathscr{L}}$, where
	\begin{equation}
	\label{definition kernel generator}
		\dot v_n^1  \colonequals 
    \realpart
		\begin{pmatrix}
			\tau^{mn-1} \\
			 i{ \frac{  (mn-1)\cOmega+\overline{\cOmega}}{mn}  \tau^{mn} } \\
			 0
		\end{pmatrix},
			\qquad
		\dot v_n^2  \colonequals 
    \realpart
    \begin{pmatrix}
      { i \tau^{mn-1} }  \\
       -{ \frac{  (mn-1)\cOmega+\overline{\cOmega}}{mn}  \tau^{mn}} \\
       0
    \end{pmatrix}.
\end{equation}
	\item \textup{(Pure rotation)} Suppose that $\imagpart{\cOmega} = 0$.  Then for any $n \geq 1$, there are exactly two values of $\Omega$ for which the dispersion relation holds:
    \begin{equation}
      \label{rotating dispersion relation}
      \Omega = \frac{\gamma}{2\pi} \left( 1 \pm \frac{1}{\sqrt{mn}} \right).
    \end{equation}
	\item \textup{(Collapsing)} If $\imagpart{\cOmega} \neq 0$, then the function $d_m(\placeholder, \cOmega)$ has no roots for any $m \geq 2$.
  \end{enumerate}
\end{lemma}
\begin{proof}
Evaluating the linearized kinematic condition \eqref{linearized kinematic condition} with $\dot \mu = 2 \realpart{(\hat \mu_n / \tau^{mn-1})}$ and $\dot \nu = 2 \realpart{( \hat \nu_\ell / \tau^{m\ell})}$, we obtain
\begin{equation}
\label{DF10} 
  \begin{aligned}
    D_{u} \F_{1}(0,0,q^0; \cOmega) \begin{pmatrix} 
      \dot \mu \\ 
      \dot \nu \\ 
      \dot q 
    \end{pmatrix}
    & = \realpart{ \left( m\ell \frac{\hat \nu_\ell}{\tau^{m\ell}} - i \cOmega \left( (1-mn) \frac{\hat \mu_n}{\tau^{mn}} +\overline{\hat \mu_n} \tau^{mn} \right)\right) } \\
    & = \realpart{\left( m\ell \frac{\hat \nu_\ell}{\tau^{m\ell}} + i \left( (mn-1) \cOmega +  \overline{\cOmega} \right) \frac{\hat \mu_n}{\tau^{mn}} \right)}.
  \end{aligned}
\end{equation}
Thus, $(\dot \mu, \dot \nu,\dot q)$ is in the kernel of the linearized kinematic condition if and only if
\[
	\hat \nu_n = -i \frac{(mn-1) \cOmega+\overline{\cOmega}}{mn} \hat \mu_n \equalscolon -i \cOmega_n \hat \mu_n \qquad \textrm{for all } n \geq 1.
\]

Similarly, evaluating the linearized Bernoulli condition acting on $(\dot \mu, \dot \nu, \dot q)$ given as above, we find that 
\begin{equation}
\label{DF20}
  \begin{aligned}
    D_{u} \F_{2}(0,0,q^0; \cOmega) \begin{pmatrix} 
      \dot \mu \\ 
      \dot \nu \\ 
      \dot q 
    \end{pmatrix}
    & = \realpart{ \left(  \left( \frac{\overline{\cgamma}}{2\pi i } + i \overline{\cOmega} \right)   \left( -m\ell \frac{\hat \nu_\ell}{\tau^{m\ell}} + \frac{\cgamma}{2\pi i} \frac{(mn-1)  \hat \mu_n} {\tau^{mn}} + i \cOmega \overline{\hat \mu_n} \tau^{mn} \right) \right)} \\
    & \qquad + |\cOmega|^2 \realpart{\left( \frac{\hat \mu_n}{\tau^{mn}}\right)}  - \dot q.
  \end{aligned}
\end{equation}
To identify the kernel of $\mathscr{L}$, we may therefore set $\hat \nu_n = -i \cOmega_n \hat \mu_n$ and evaluate along the diagonal $\ell = n$, which gives
\begin{align*}
	0 & = \realpart{ \left(  \left( \frac{\overline{\cgamma}}{2\pi i } + i \overline{\cOmega} \right)   \left( mn \frac{i \cOmega_n \hat \mu_n}{\tau^{mn}} + \frac{\cgamma}{2\pi i} \frac{(mn-1)  \hat \mu_n} {\tau^{mn}} + i \cOmega \overline{\hat \mu_n} \tau^{mn} \right) \right)}\\
		& \qquad + |\cOmega|^2 \realpart{\left( \frac{\hat \mu_n}{\tau^{mn}}\right)}  - \dot q \\
		& = \realpart{ \left(\left( \left( \frac{\overline{\cgamma}}{2\pi i } + i \overline{\cOmega} \right)\left( mn i \cOmega_n + \frac{\cgamma}{2\pi i} (mn-1) \right) +  \frac{\cgamma}{2\pi} \overline{\cOmega}   \right) \frac{\hat \mu_n}{\tau^{mn}} \right)} - \dot q \\ 
		& = \realpart{ \left( \left(  (1-mn) \left| \frac{\cgamma}{2\pi i} + i \cOmega \right|^2 + \frac{\gamma}{\pi} \overline{\cOmega} - \overline{\cOmega}^2  \right) \frac{\hat \mu_n}{\tau^{mn}} \right)} -\dot q \\
		& = \realpart{ \left( d_{m}(n,\cOmega) \frac{\hat \mu_n}{\tau^{mn}} \right)} - \dot q,
\end{align*}
where $d_m$ is given by \eqref{definition dispersion relation}.   Note that here we have made use of the fact that $\imagpart{\cgamma}/2\pi = \imagpart{\cOmega}$ to see that only $\gamma$ occurs in the dispersion relation.  Clearly, then, the kernel is nontrivial if and only if $d_m(n,\cOmega) = 0$, and we can see from this construction that it is  generated by $\{\dot v_n^1, \dot v_n^2\}$.   Moreover, $d_m(\placeholder,\cOmega)$ is strictly monotone for $\cOmega \neq 0$, so there can be at most one root for each fixed $m$ and $\cOmega$.  This proves the statement in part~\ref{dispersion relation part}.

 If $\imagpart{\cOmega} = 0$, meaning we have purely rotating (non-collapsing) vortices, then it is easy to see that 
\[
	d_m(n,\cOmega) = d_m(n,\Omega) = (1-mn) \left(\Omega-\frac{\gamma}{2\pi} \right)^2 + \frac{\gamma}{\pi} \Omega.
\]
Setting $d_m(n,\Omega) = 0$, gives a quadratic polynomial in $\Omega$ whose roots are precisely those in~\eqref{rotating dispersion relation}. On the other hand, if $\imagpart{\cOmega} \neq 0$, then we calculate that
\[
	\imagpart{d_m(n,\cOmega)} = \frac{1}{\kappa} \left( \frac{\gamma}{2\pi} - \Omega \right),
\]
and hence the kernel will be nontrivial only if $\Omega = \gamma/(2\pi)$.  Taking this value for $\Omega$, we then find that 
\[
	\realpart{d_m(n,\cOmega)} =  \frac{\gamma}{2\pi} \Omega - \Omega^2 + \frac{1}{4 \kappa^2} = \frac{1}{4 \kappa^2}.
\]
But the right-hand side above is clearly non-vanishing when $\imagpart{\cOmega} \neq 0$, so $d_m(\placeholder, \cOmega)$ has no roots for the collapsing case.
\end{proof}

Consider next the range of $\mathscr{L}$.

\begin{lemma}[Range]
\label{range lemma}
Fix $m \geq 2$ and $\cOmega \in \mathbb{C}$.  
\begin{enumerate}[label=\rm(\alph*\rm)]
	\item If $d_m(n,\cOmega) \neq 0$ for all $n \geq 1$, then $\mathscr{L}$ is surjective.  
	\item \label{cokernel generator part} If $d_m(n,\cOmega) = 0$ for some $n \geq 1$, then for  $\dot w_n^1, \dot w_n^2 \in \Yspace^*$ given by
	\begin{equation}
		\label{range condition}
    \begin{aligned}
      \langle \dot w_n^1, ~(\A,\B) \rangle_{\Yspace^* \times \Yspace} & \colonequals \realpart \int_{\mathbb{T}} \left( \left( \frac{\overline{\cgamma}}{2\pi i} + i \overline{\cOmega} \right) \A + \B \right) \tau^{mn} \, d\theta,
      \\
      \langle \dot w_n^2,~(\A,\B) \rangle_{\Yspace^* \times \Yspace} & \colonequals \realpart \int_{\mathbb{T}} \left( \left( \frac{\overline{\cgamma}}{2\pi i} + i \overline{\cOmega} \right) \A + \B \right) i \tau^{mn} \, d\theta .
    \end{aligned}
\end{equation}
it holds that $\range{\mathscr{L}} = \kernel{\dot w_n^1} \cap \kernel{\dot w_n^2}$.  In particular, $\range{\mathscr{L}}$ has codimension $2$.
\end{enumerate}
\end{lemma}
\begin{proof}
Let $(\A,\B) \in \Yspace$ be given and suppose that $\mathscr{L}(\dot \mu ,\dot \nu, \dot q) = (\A,\B)$ for some $(\dot \mu, \dot \nu, \dot q) \in \Xspace \times \mathbb{R}$.  Writing
\[
	\dot \mu(\tau) = 2\realpart \sum_{n \geq 1} \frac{\hat \mu_n}{\tau^{mn-1}}, \quad \dot \nu(\tau) = 2 \realpart \sum_{n \geq 1} \frac{\hat \nu_n}{\tau^{mn}}, \quad \A(\tau) = 2 \realpart \sum_{n \geq 1} \frac{\hat a_n}{\tau^{mn}}, \quad \B(\tau) = 2\realpart \sum_{n \geq 0} \frac{\hat b_n}{\tau^{mn}},
\]
then the linearized kinematic condition gives
\[
	m\ell \hat \nu_\ell + i \cOmega_\ell \hat \mu_\ell  = \hat a_\ell \qquad \textrm{for all } \ell \geq 1.
\]
Thus we can ``row-reduce'' $\mathscr{L}$ to eliminate $\dot \nu$ in favor of $\A$ and $\dot \mu$.  Making this substitution in the linearized dynamic condition and computing as in the proof of Lemma~\ref{null space lemma} leads to
\begin{align*}
	\hat b_\ell &= 
		\left\{ \begin{aligned}
			d_m(\ell,\cOmega) \hat \mu_\ell  - \left( \frac{\overline{\cgamma}}{2\pi i}+ i \overline{\cOmega} \right) \hat a_\ell & \qquad \textrm{for all } \ell \geq 1 \\
			-\dot q & \qquad \textrm{for } \ell = 0.
		\end{aligned} \right.
\end{align*}
Clearly, this implies that $\mathscr{L}$ is surjective if $d_m(\ell,\cOmega) \neq 0$ for all $\ell \geq 1$, which corresponds to $\mathscr{L}$ being injective by Lemma~\ref{null space lemma}.  
By the same result, we know that there exists at most one $n \geq 1$ such that $d_m(n,\cOmega) = 0$.  In this case, we must have that
\[
	 \left( \frac{\overline{\cgamma}}{2\pi i} + i \overline{\cOmega} \right) \hat a_n + \hat b_n= 0,
\] 
which is equivalent to \eqref{range condition}.
By the continuity of the index, $\mathscr{L}$ continues to be Fredholm index zero, and so this is a complete characterization of $\range{\mathscr{L}}$.
\end{proof}

\subsection{Proof of Theorems~\ref{intro m-fold rotating theorem} and~\ref{intro rigidity corollary}}

We are now prepared to prove the existence of the nearly circular $m$-fold symmetric rotating hollow vortices and the local uniqueness of circular collapsing vortices.  First, consider the rotating case.  

\begin{theorem}[Rotating hollow vortices] \label{m-fold rotating theorem}
Fix $m \geq 2$ and let $\Omega^0 = \Omega_{0,\pm}$, for $\Omega_{0,\pm}$ defined as in~\eqref{definition Omega_0^pm}.  There exists a one-parameter family $\cm_{\rot}^{m,\pm} \subset \mathscr{O}$ of solutions to the rotating hollow vortex problem admitting the real-analytic parameterization   
\[
	\cm_{\rot}^{m,\pm} = \left\{ (\mu_\pm^\varepsilon, \nu_\pm^\varepsilon, q_\pm^\varepsilon, \Omega_\pm^\varepsilon ) : |\varepsilon| \ll 1 \right\} \subset \F^{-1}(0)
\] 
with
\begin{equation}\label{rotating density asymptotics}
	\mu_\pm^\varepsilon = 2\varepsilon \realpart{\frac{1}{\tau^{m-1}}} + O(\varepsilon^2), \quad  \nu_\pm^\varepsilon = 2\varepsilon \realpart{\frac{\Omega_{0,\pm}}{\tau^{m}}} + O(\varepsilon^2) \qquad \textrm{in } C^{k+1+\alpha}(\mathbb{T}),
\end{equation}
and
\[
	q_\pm^\varepsilon = q^0(\Omega_{0,\pm}) + O(\varepsilon^2), \quad \Omega_\pm^\varepsilon = \Omega_{0,\pm} + O(\varepsilon^2).
\]
In a neighborhood of $(0, \Omega_{0,\pm})$, the curve $\cm_\rot^{m,\pm}$ comprises all nontrivial solutions of $\F(u,\Omega) = 0$.  
\end{theorem}
\begin{remark}
We derive higher-order asymptotics for these solutions in Appendix~\ref{asymptotics appendix}. Note also that uniqueness holds in a larger symmetry class. The dispersion relation function satisfies  $d_m(n,\placeholder) = d_n(m,\placeholder)$ for any $n \geq 1$. Letting $\ell \geq 2$ be the largest prime divisor of $m$, the same argument given below reveals that the solutions along $\cm_{\rot}^{m,\pm}$ are locally the only nontrivial solutions in the larger space of $\ell$-fold symmetric vortices. 
\end{remark}

\begin{proof}[Proof of Theorem~\ref{m-fold rotating theorem}]
  For $\Omega^0$ given as above, we have by \eqref{definition Omega_0^pm} and \eqref{definition dispersion relation} that $d_m(1,\Omega^0) = 0$.  Define a new mapping $\G = \G(u; \Omega) \colon \mathscr{O}_\even  \to \Yspace_\even$ by
\[
	\G(\mu,\nu,q; \Omega) \colonequals \F(\mu,\nu,q+q^0(\Omega); \Omega).
\]
Thus, 
\[
	\G(0,0,0; \Omega) = 0 \qquad \textrm{for all } \Omega \in \mathbb{R}.
\]
Moreover, because 
\[
	D_{u} \G(0,0,0; \Omega) \equalscolon \G_{u}^0 =  \mathscr{L} \colon \Xspace_\even \to \Yspace_\even,
\]
  we have by Lemmas~\ref{null space lemma} and \ref{range lemma}, that $\G_u^0 $ is Fredholm index $0$ with a one-dimensional null space generated by $\dot v_1^1$ and cokernel generated by $\dot w_1^1$.  Here we have used that $\dot v_1^2$ is not in $\Xspace_\even$ and $\dot w_1^2|_{\Yspace_\even} = 0$.

It remains now only to check the transversality condition Theorem~\ref{CR theorem}\ref{CR transversality}. But from \eqref{DF10} and \eqref{DF20} it follows that
\begin{align*}
	\G_{\Omega u }^0 \dot v_1 & = \F_{\Omega u}^0 \dot v_1 + (\partial_\Omega q^0)(\Omega_{0,\pm}) \F_{q u}^0 \dot v_1 
		 = 
     \realpart
			\begin{pmatrix}
				  \frac{mi}{\tau^{m}} 	\\
				 \frac{\partial_{\Omega} d_m(1,\Omega_{0,\pm})}{\tau^{m}} 
			\end{pmatrix} 
		\equalscolon
			\begin{pmatrix}
				\A \\
				\B
			\end{pmatrix}.
\end{align*}
Thus, in view of the range condition~\eqref{range condition}, we consider 
\begin{align*}
	\frac{1}{2\pi } \int_{\mathbb{T}} \left( \left( \frac{\gamma}{2\pi i} + i \Omega_{0,\pm} \right) \A + \B \right) \tau^{m} \, d\theta &=   m \left(\frac{\gamma}{2\pi } - \Omega_{0,\pm} \right)   + (\partial_\Omega d_m)(1,\Omega_{0,\pm}) \\
		& =  m \left(\frac{\gamma}{2\pi } - \Omega_{0,\pm} \right) -  m \left( \frac{\gamma}{\pi} - 2\Omega_{0,\pm}\right) \\
		& = 3m\left(\frac{\gamma}{2\pi} -  \Omega_{0,\pm} \right) = \mp  \frac{3\sqrt{m}  \gamma}{2\pi},
\end{align*}
  where the last equality is from the formula for $\Omega_{0,\pm}$ in~\eqref{definition Omega_0^pm}. As the right-hand side is non-vanishing, the existence of the local curve $\cm_\rot^{m,\pm}$ is an immediate consequence of the Crandall--Rabinowitz theorem.  Likewise, the leading-order form of the solutions is found from \eqref{definition kernel generator}. Here we have used the symmetries of the equation and uniqueness to infer that $\Omega_\pm^\varepsilon,q_\pm^\varepsilon$ are even in $\varepsilon$, and hence agree with $\Omega_{0,\pm},q^0(\Omega_{0,\pm})$ to $O(\varepsilon^2)$.
\end{proof}

Finally, we turn our attention to the proof of Theorem~\ref{intro rigidity corollary}.  A key technical ingredient is the next lemma, which guarantees that all self-similarly collapsing $m$-fold symmetric hollow vortices with appropriate behavior at infinity and sufficiently close to a circle can, potentially following a rescaling of space, be written using our ansatz for $(w,f)$ in~\eqref{w and f ansatz}.

\begin{lemma}[Ansatz]
\label{uniqueness lemma}
Suppose that there is a single self-similarly imploding hollow vortex with vortex core $\mathscr{V}$ that has $C^{1+\alpha}$ boundary and contains $0$.  Assume also that the vortex is $m$-fold symmetric and the relative velocity vanishes at infinity.  Let $\varrho > 0$ and $f^\varrho \colon \confD \to \mathbb{C}$ be the (unique) univalent conformal mapping with 
\begin{equation}
  \label{f rho definition}
  \fluidD = f^\varrho(\confD), \quad f^\varrho(\infty) = \infty, \quad \textrm{and} \quad  f^\varrho_\zeta(\infty) = \varrho >  0.
\end{equation}
Then $f^\varrho$ and the corresponding complex potential $w$ satisfy
\begin{equation}
  \label{arbitrary vortex description}
  f^\varrho|_{\mathbb{T}} = \varrho \left( \id{} + \mathcal{C} \mu \right), \qquad w|_{\mathbb{T}} =  \left( \frac{\gamma}{2\pi i} + \varrho^2 \Omega \right) \log{(\placeholder)} + \mathcal{C} \nu 
\end{equation}
for a pair of densities $(\mu,\nu) \in \Xspace$.
\end{lemma}
\begin{proof}
Let $f^\varrho$ and $\varrho$ be given as in the statement of the lemma. 
Since by assumption $\mathscr{V}$ is a $C^{1+\alpha}$ domain, Kellogg's theorem guarantees that $f^\varrho \in C^{1+\alpha}(\overline{\confD})$.

The first task is to rescale so that the conformal mapping is normalized as in our ansatz~\eqref{w and f ansatz}.  
Setting $f \colonequals f^\varrho/\varrho$, it follows that $f \in C^{1+\alpha}(\overline{\confD})$, $f$ is conformal on $\confD$, and
\[
	{f}(\confD) = \tfrac{1}{\varrho} \fluidD, \quad {f}(\infty) = \infty, \quad \textrm{and} \quad  {f}_\zeta(\infty) = 1.
\]
Moreover, we see that $(f, w, \varrho^2 \cOmega, \gamma, \varrho^2 q)$ is a solution to the conformal formulation of the hollow vortex problem~\eqref{conformal governing equations}.
The $m$-fold symmetry of the domain in the $z$-plane is equivalent to this rescaled $f$ exhibiting the symmetry~\eqref{m-fold symmetry} in the $\zeta$-plane.  We infer that the Laurent expansion of $f - \id$ must be as in \eqref{w and f ansatz} for some $\{ \hat f_n\} \subset \mathbb{C}$.

Likewise, the  potential necessarily has the form 
\[
	w(\zeta) =  \frac{a}{2\pi i} \log{\zeta} + \sum_{n \geq 1} \frac{b_n}{\zeta^n},		
\]
for some $a, b_n \in \mathbb{C}$. Note that here we are free to set the constant term to be $0$.  The circulation condition~\eqref{circulation condition} then forces $\realpart{a} = \gamma$, while the $m$-fold symmetry~\eqref{m-fold symmetry} implies that $b_{n} = 0$ whenever $n$ is not divisible by $m$.  Setting $\hat w_j \colonequals b_{jm}$, we arrive at a similar Laurent expansion as in~\eqref{w and f ansatz}; it remains only to determine $\imagpart{a}$. But, from the kinematic condition~\eqref{kinematic condition}, it follows that
\begin{align*}
	0 & = \realpart{\left( \zeta w_\zeta +  i \varrho^2 \cOmega \zeta f_\zeta \overline{f} \right)} \\ 
		& = \imagpart{\left(\frac{a}{2\pi} - \varrho^2 \cOmega \right)}  + \realpart{\left( \zeta w_\zeta - \frac{a}{2\pi i} + i \varrho^2 \cOmega \left( (f_\zeta -1)(\overline{f -\id}) + \overline{f - \id} + \overline{\zeta} (f_\zeta - 1) \right) \right)}
	\qquad \textrm{on } \mathbb{T}.
\end{align*}
Note that the real part above has no constant terms, and hence 
\[
	\imagpart{a} = 2\pi \varrho^2 \imagpart{\cOmega} = 2\pi \varrho^2 \Omega.
\]

Now, $w - (a/2\pi i) \log{\zeta}$ and $f-\id$ are single-valued and holomorphic function on $\confD$ and of class $C^{1+\alpha}(\overline{\confD})$. Applying the Sokhotski--Plemelj formula, we see that their traces on $\mathbb{T}$ can each be written using the layer-potential representation~\eqref{m-fold trace ansatz} for the real-valued densities $\mu, \nu \in C^{1+\alpha}(\mathbb{T})$ given explicitly by
\[
	\mu(\tau) = -2 \realpart{\sum_{n \geq 1} \frac{\hat f_n}{\tau^{mn-1}}} \qquad \nu(\tau) = -2 \realpart{\sum_{n \geq 1} \frac{\hat w_n}{\tau^{mn}}}.
\]
The claimed forms of $f^\varrho$ and $w$ now follow from the value of $a$ derived above.
\end{proof}

The next lemma establishes that, when the vortex core is sufficiently close to circular, the corresponding conformal mapping is near-identity in $C^{1+\alpha}$. Results of this type are of course extremely classical. For example, Marchenko~\cite{marchenko1935representation} provides a bound for $f-\id$ in $W^{1,p}$ with $p$ in a certain range.  We will use a sharpened version of that theorem due to Gaier~\cite{gaier1962conformal}, then upgrade it to $C^{1+\alpha}$ control using a simple Schauder theory argument.

\begin{lemma}[Near-circular domain] 
\label{near identity lemma}
Let $\Gamma$ be a Jordan curve that is star shaped with respect to the origin, and suppose that it has a parameterization $R(\theta) e^{i\theta}$ for $\theta \in [0,2\pi]$, such that $R$ is $2\pi$-periodic, $R \in C^{1+\alpha}(\mathbb{T})$, and  
\begin{equation}
  \label{near circular condition}
  \| R - 1 \|_{C^{1+\alpha}} \leq \varepsilon.
\end{equation}
Let $\fluidD$ be the region exterior to $\Gamma$, and suppose $\varrho > 0$ and the conformal mapping $f^\varrho \colon \confD \to \fluidD$ are given as~\eqref{f rho definition}. Then for $\varepsilon$ sufficiently small, we have
\begin{equation}
  \label{near identity bound}
  \| f^\varrho - \varrho \id \|_{C^{1+\alpha}(\mathbb{T})} \leq C \varepsilon,
\end{equation}
where the constant $C > 0$ depends on upper bounds on $\alpha$ and $\varepsilon$.
\end{lemma}
\begin{proof}
Let $\Gamma$ have the $C^{1+\alpha}$ parameterization $R(\theta) e^{i\theta}$ as stated above. We assume that $\varrho = 1$ and simply write $f$ for the conformal mapping $\confD \to \fluidD$; the general case follow from a simple rescaling as in the proof of Lemma~\ref{uniqueness lemma}. Throughout the argument, we let $\varepsilon_0 > 0$ denote an upper bound for $\varepsilon$ and write $C_{\varepsilon_0} > 0$ for a generic constant that depends only on $\varepsilon_0$. As $\alpha$ is fixed, all dependencies on it will be suppressed. 

Let $g \colon \mathbb{D} \to \mathbb{C}$ be the holomorphic function defined by $g(\zeta) \colonequals f(1/\zeta)$. Thus $g$ is a conformal mapping satisfying 
\[
	\mathscr{V} = g(\mathbb{D}), \qquad g(0) = 0, \qquad g^\prime(0) = 1,
\]
where $\mathscr{V}$ denotes the interior region of $\Gamma$.  Moreover, by Kellogg's theorem, $g$ extends to a $C^{1+\alpha}$ function on $\overline{\mathbb{D}}$. Clearly, an estimate of the form~\eqref{near identity bound} for $g-\id$ will imply the desired estimate for $f - \id$. Next, we set
\[
 	\phi \colonequals \arg{g}, \qquad r \colonequals |g|.
\]
Then $\phi$ is harmonic in $\mathbb{D} \setminus \{ \zeta \leq 0\}$ and, again by Kellogg's theorem, $\phi$ has a $C^{1+\alpha}$ extension to $\overline{\mathbb{D}} \setminus \{ \zeta \leq 0\} $. Likewise,  $\log{r}$ is the harmonic conjugate of $\phi$ and admits a $C^{1+\alpha}$ extension to $\overline{\mathbb{D}} \setminus \{0\}$.  The main point is that, while these functions are $C^{1+\alpha}$, we need quantitative bounds of the form~\eqref{near identity bound} on $\phi - \arg{(\placeholder)}$ and $r - 1$ in $C^{1+\alpha}(\mathbb{T})$. 

Since $\theta \mapsto g(e^{i\theta})$ and $\theta \mapsto R(\theta) e^{i\theta}$ are both parameterizations of $\Gamma$, equating them yields the identity
\begin{equation}
  \label{two parameterizations identity}
  R(\phi(\tau)) = r(\tau) \qquad \textrm{for all } \tau \in \mathbb{T}.
\end{equation}
Taking a tangential derivative $\partial_\theta = -i \bar\tau \partial_\tau$ and dividing by $r$, we therefore have
\begin{equation}
  \label{theta derivative parameterization identity}
  \frac{R^\prime(\phi)}{r} \partial_\theta \phi =  \frac{\partial_\theta r}{r} =  \partial_{\mathbf{n}} \phi \qquad \textrm{on } \mathbb{T},
\end{equation}
where $\partial_{\mathbf{n}}$ is the outward unit normal derivative along $\mathbb{T}$ and we have used the fact that $\phi$ and $\log{r}$ are conjugates. Let $\varphi \colonequals \phi - \arg{(\placeholder)}$. Then $\varphi$ is single-valued and harmonic on $\mathbb{D}$, and it admits a $C^{1+\alpha}$ extension to $\overline{\mathbb{D}}$. As a consequence of the near circular assumption~\eqref{near circular condition}, 
\begin{equation}
  \label{epsilon condition}
  \left\| \frac{R^\prime}{R} \right\|_{C^0} \leq \varepsilon, \qquad \| R - 1 \|_{C^0} < \varepsilon.
\end{equation}
Thus, by a theorem of Marchenko~\cite{marchenko1935representation},  $\phi$ is uniformly bounded in $W^{1,\frac{\alpha}{1+\alpha}}(\mathbb{T})$ in terms of $\varepsilon_0$. Specifically, using the sharp version due to Gaier~\cite{gaier1962conformal}: 
\begin{equation}
  \label{gaier bounds}
  \begin{aligned}
    \| \phi^\prime  \|_{L^{\frac{\alpha}{1+\alpha}}(\mathbb{T})} & \leq 4 \left( \frac{2 \pi }{\cos{( \frac{\alpha}{1+\alpha} \| \vartheta \|_{C^0(\mathbb{T})})}} \right)^{1+\frac{1}{\alpha}} \qquad \textrm{if } \frac{\alpha}{1+\alpha} < \frac{\pi}{2 \| \vartheta \|_{C^0(\mathbb{T})}}, \\
    \| \varphi \|_{C^0(\mathbb{T})} & \leq \frac{\pi}{\sqrt{3}} \frac{\varepsilon}{1-\varepsilon};
  \end{aligned}
\end{equation}
see~\cite[Theorem 6 and (1.6)]{gaier1962conformal}.  Here, $\vartheta \colonequals -\arctan{(R^\prime/R)}$ denotes the angle of inclination of the tangent to $\Gamma$; observe that $\|\vartheta\|_{C^0} < C_{\varepsilon_0} \varepsilon$ in light of~\eqref{near circular condition}. A proof of Gaier's bounds~\eqref{gaier bounds} in terms of $\vartheta$ can be found in~\cite[Lemma 6.2]{hassainia2020global}, for example. From~\eqref{gaier bounds}, Morrey's inequality, the near-circularity of $\Gamma$~\eqref{near circular condition}, and the relation~\eqref{two parameterizations identity}, it follows that for $\varepsilon_0$ sufficiently small,
\begin{equation}
  \label{r Holder bounds}
  \| \varphi \|_{C^{\alpha}(\mathbb{T})} \leq C_{\varepsilon_0}, \qquad \| \varphi \|_{C^0(\mathbb{T})} \leq C_{\varepsilon_0} \varepsilon, \qquad
  \| r - 1 \|_{C^\alpha} \leq C_{\varepsilon_0} \varepsilon.
\end{equation}

Rearranging slightly the differentiated identity~\eqref{theta derivative parameterization identity}, we see that $\varphi$ satisfies an inhomogeneous uniformly oblique boundary condition:
\[
	\partial_{\mathbf{n}} \varphi -  \frac{R^{\prime}(\phi)}{r} \partial_\theta \varphi = -\frac{R^\prime(\phi)}{r}  \qquad \textrm{on } \mathbb{T}.
\]
Thanks to~\eqref{near circular condition} and ~\eqref{r Holder bounds}, the coefficients and the data above are bounded in $C^{\alpha}(\mathbb{T})$ uniformly in $\varepsilon$ for $0 \leq \varepsilon \ll 1$. Applying a Schauder estimate for weak solutions~\cite[Theorem 3]{constantin2011discontinuous}, we infer that
\begin{equation}
  \label{near circular varphi C1alpha estimate}
  \begin{aligned}
    \| \varphi \|_{C^{1+\alpha}(\mathbb{D})} & \leq C_{\varepsilon_0} \left( \| \varphi \|_{C^0(\mathbb{D})} + \| R^\prime \circ \phi\|_{C^\alpha(\mathbb{T})} \right) \leq C_{\varepsilon_0} \left( \| \varphi \|_{C^0(\mathbb{T})} + \varepsilon \right) \\
    & \leq C_{\varepsilon_0} \varepsilon.
  \end{aligned}
\end{equation}
Note that in the last line we used~\eqref{gaier bounds} to control the boundary values of $\varphi$. Returning to the identity~\eqref{two parameterizations identity}, the above estimate and \eqref{near circular condition} furnish an improved $C^{1+\alpha}$ bound for $r$:
\[
	\| r - 1 \|_{C^{1+\alpha}(\mathbb{T})} \leq C_{\varepsilon_0} \varepsilon.
\]
Taken together with the $C^{1+\alpha}$ estimate~\eqref{near circular varphi C1alpha estimate} for $\varphi$, this yields the claimed bound \eqref{near identity bound}.
\end{proof}

\begin{proof}[Proof of Theorem~\ref{intro rigidity corollary}]
As in the proof of Lemma~\ref{null space lemma}, we see that the null space of $\mathscr{L}$ is trivial.  By the implicit function theorem, there exists a neighborhood $\mathcal{O}$ of $(0,0,q^0,\cOmega)$ such that $(\mu,\nu, q,\cOmega) \in \mathcal{O}$ is in the zero-set of $\F$ if and only if $(\mu,\nu,q) = (0,0,q^0(\cOmega))$. 

On the other hand, by Lemma~\ref{uniqueness lemma}, the self-similarly collapsing hollow vortex $\mathscr{V}$ admits a description via the layer-potential ansatz~\eqref{arbitrary vortex description}.  Performing a rescaling as in the proof of that lemma, we may without loss of generality take $\varrho = 1$. Let $Z$ be the given parameterization of the vortex boundary, and write $Z(\theta) \equalscolon R(\theta) e^{i\theta}$ for $\theta \in [0,2\pi)$.  Then, by assumption, 
\[
	\| R - 1 \|_{C^{1+\alpha}} \leq \varepsilon.
\]
From Lemma~\ref{near identity lemma}, we conclude that $\| \mu \|_{C^{1+\alpha}} \leq C \varepsilon$ for some $C>0$ depending on an upper bound of $\varepsilon$. But then, the kinematic condition $\F_1(\mu,\nu,q) = 0$ implies that
\begin{align*}
	\realpart{\left(\frac{\cgamma}{2\pi i} + i \cOmega \right)} + \realpart{\left( \tau \mathcal{C} \nu^\prime\right)}& = O(\varepsilon) \qquad \textrm{in } C^\alpha(\mathbb{T}).
\end{align*}
Since $\realpart{ \tau \mathcal{C} \partial_\tau}$ is a bijective mapping $\mathring{C}^{1+\alpha}(\mathbb{T}) \to \mathring{C}^\alpha(\mathbb{T})$, by applying the spectral projections $P_0$ and $1-P_0$, we see that 
\[
	|\cgamma - \cgamma^0(\gamma, \cOmega)| + \| \nu \|_{C^{1+\alpha}} = O(\varepsilon).
\]
Finally, the above bounds and the Bernoulli condition imply that $q - q^0(\cOmega)$ is likewise $O(\varepsilon)$. Thus, for $\varepsilon$ sufficiently small, $(\mu,\nu,q,\cOmega) \in \mathcal{O}$, and hence the rescaled vortex is exactly the unit ball. It follows that the unscaled vortex $\mathscr{V}$ is likewise a ball, and so the fact that its area is $\pi$ forces it to be $B_1$.
\end{proof}

\section{Multiple imploding hollow vortices} 
\label{multiple collapsing section}

\subsection{Self-similar point vortex configurations}
\label{self-similar point vortex section}

Let us recall briefly the governing equations for self-similar collapsing point vortex configurations. Suppose there is a collection of $M \geq 3$ point vortices in the plane, which we identify with $\mathbb{C}$ in the usual way.  (It is not difficult to prove that there are no collapsing configurations consisting of two point vortices.) Let $z_1(t), \ldots, z_M(t) \in \mathbb{C}$ denote the position of the vortex centers at time $t$, with the corresponding circulations being $\gamma_1, \ldots, \gamma_M$.  
The Kirchhoff--Helmholtz model states that the vortex centers evolve according to the ODE system
\[
	\partial_t \overline{z_k(t)} = \sum_{j \neq k} \frac{\gamma_j}{2 \pi i} \frac{1}{z_k(t) - z_j(t)}.
\]
Further background can be found in \cite{saffman1992book,newton2001nvortex}.  There are many well-known \emph{steady} or \emph{equillibrium} solutions to this system for which the vortex centers appear at rest when viewed in a rigidly translating or rotating reference frame;  see, for example, \cite{aref2003crystals,aref2007playground,newton2014postaref} and the references therein.  Our focus, though, will be configurations that experience  \emph{finite-time collapse}: there exists $\kappa > 0$ and $\zc \in \mathbb{C}$ such 
\[
	z_k(t) \longrightarrow \zc \qquad \textrm{as } t \nearrow \kappa,~\textrm{for all } 1 \leq k \leq M.
\]
Observe that the \emph{linear impulse} 
\[
	\mathcal{I} \colonequals \sum_k \gamma_k z_k(t)
\]
is a conserved quantity.  Thus, the only potential collapse point must be
\begin{equation}
  \label{definition zc}
  \zc = \frac{\mathcal{I}}{\sum_k \gamma_k} = \frac{\sum_k \gamma_k z_k(0)}{\sum_k \gamma_k}.
\end{equation}
One may therefore view $\zc$ as a function of the initial vortex center positions.  

We say that a configuration experiences \emph{self-similar collapse}  to $\zc$ provided that 
\begin{equation}
  \label{definition self-similar point vortex configuration}
   z_k(t)  = \zc+\boldsymbol{L}(t) (z_k(0) - \zc),
\end{equation}
for some time-dependent scale factor $\boldsymbol{L}(t) \in \mathbb{C}$. One can show that for finite-time collapse the only possibility is 
\[
	\boldsymbol{L}(t) \partial_t \overline{\boldsymbol{L}(t)} = -\frac{1}{2\kappa} - i \Omega \equalscolon -i \cOmega,
\]
for a constant $\cOmega \in \mathbb{C}$ with nonzero imaginary part;  see \cite[Section~2.1]{newton2001nvortex}. Here $\cOmega$ is the constant from \eqref{cOmega definition}, and the
configurations will appear steady in the $(\boldsymbol \xi,T)$ coordinate system from \eqref{definition T xi}.
For vortex triples ($M=3$), Krishnamurthy and Stremler~\cite{krishnamurthy2018finite} prove that every collapsing configuration does so self-similarly, but to the best of our knowledge, this has not been established for general $M$.

From the discussion above, it follows that self-similarly collapsing point vortex configurations correspond to solutions of the system of $M$ algebraic (complex) equations
\begin{equation}
  \label{vortex V tilde definition}
  0 = \widetilde{\mathcal{V}}_k(z_1,\ldots, z_M, \gamma_1, \ldots, \gamma_M, \cOmega) \colonequals \sum_{j \neq k} \frac{\gamma_j}{2 \pi i} \frac{1}{z_k - z_j} + i\cOmega \left(\overline{{z}_k-{\zc} }\right),
\end{equation}
for $k = 1, \ldots, M$.  Here $z_1, \ldots, z_M$ now represent the initial positions of the vortex centers; the time-dependent positions are uniquely determined by \eqref{definition self-similar point vortex configuration}.  Let 
\[
	\Lambda \colonequals (z_1, \ldots, z_M, \gamma_1, \ldots, \gamma_M, \cOmega) \in \mathbb{C}^M \times \mathbb{R}^M \times \mathbb{C} \equalscolon \mathbb{P}
\]
denote the full set of parameters for the point vortex configuration.  The collapse point $\zc$ must then be the function of $\Lambda$ given by \eqref{definition zc}. However, it can be directly confirmed that the system \eqref{vortex V tilde definition} has the appealing structural feature that 
\[
	\left\{ \begin{aligned}
		\zc(\Lambda) & = 0 \\
		\widetilde{\mathcal{V}}(\Lambda) &= 0
	\end{aligned} \right. \qquad \textrm{if and only if} \qquad \mathcal{V}(\Lambda) = 0,
\]
where $\mathcal{V} = (\mathcal{V}_1, \ldots, \mathcal{V}_M) \colon \mathbb{P} \to \mathbb{C}^M$ is defined by
\begin{equation}
  \label{vortex V definition}
  \mathcal{V}_k = \mathcal{V}_k(\Lambda) \colonequals \sum_{j \neq k} \frac{\gamma_j}{2 \pi i} \frac{1}{z_k - z_j} + i \cOmega \overline{z}_k \qquad k = 1, \ldots, M.
\end{equation}
It is therefore equivalent to study the simpler zero-set of $\mathcal{V}$, with the understanding that this amounts to choosing coordinates so that the vortices collapse to the origin.

Our ultimate goal is to ``desingularize'' point vortex configurations of this type into imploding hollow vortex configurations.  As is typical for such arguments, this will require that the starting point vortex solution is non-degenerate in an appropriate sense.  With that in mind, we make the following definition.

\begin{definition}[Non-degeneracy]
\label{definition non-degeneracy}
A point vortex configuration $\Lambda_0 \in \mathbb{P}$ is said to be \emph{non-degenerate} provided that $\mathcal{V}(\Lambda_0) = 0$ and
\[
	D_\Lambda \mathcal{V}(\Lambda_0) \quad \textrm{is full rank}.
\]
Alternatively, we can identify $\mathbb{P} \cong \mathcal{P} \times \mathcal{P}^\prime$, for a $2M$-dimensional subspace $\mathcal{P}$ of $\mathbb{P}$ with complementary subspace $\mathcal{P}^\prime$.  Writing $\Lambda = (\lambda, \lambda^\prime)$, non-degeneracy holds provided that
\[
	D_\lambda \mathcal{V}(\Lambda_0) \quad \textrm{is invertible}.
\]
\end{definition}

By the implicit function theorem, whenever $\Lambda_0$ is non-degenerate, there exists a $(M+2)$-dimensional manifold of nearby collapsing point vortex configurations that can be expressed as a graph over $\mathcal{P}^\prime$.

\begin{example}[Vortex triples] \label{collapsing triple example}
The construction of collapsing point vortex triples ($M=3$) was first accomplished by Gröbli \cite{grobli1877specielle} as part of his 1877 doctoral thesis; see \cite{aref1992grobli,goodman2024translation} for a translation and historical discussion of this work.  Nearly a century later, Gröbli's result was rediscovered independently by Aref~\cite{aref1979motion} and Novikov and Sedov~\cite{novikov1979vortex}.  Many other authors later contributed to the theory of collapsing vortex triples; see~\cite{aref2010self} for a discussion of these works.  

Following the presentation of Aref \cite[Section B]{aref2010self}, we find that we can locally parameterize the solutions of \eqref{vortex V tilde definition} via $(\gamma_1, \gamma_2, d, \theta) \in \mathbb{R}^2 \times \mathbb{R}_+ \times [0,2\pi)$ by taking the remaining parameters according to
\[
\begin{gathered}
	z_1  \colonequals -\frac{\gamma_2 d}{\gamma_1+\gamma_2}, \qquad  
	z_2  \colonequals \frac{\gamma_1 d}{\gamma_1+\gamma_2}, \qquad 
	\gamma_3  \colonequals -\frac{\gamma_1\gamma_2}{\gamma_1+\gamma_2}, \\ 
	z_3 \colonequals d e^{i\theta} \sqrt{\frac{\gamma_1+ \gamma_2+\gamma_3}{\gamma_1+\gamma_2}},  \qquad
	\Omega \colonequals \frac{1}{4\pi} \left( \frac{\gamma_2+\gamma_3}{|z_2-z_3|^2} + \frac{\gamma_3+\gamma_1}{|z_3-z_1|^2} + \frac{\gamma_1+\gamma_2}{d^2} \right),  \\
	\frac{\pi^2}{\kappa^2} \colonequals -4\pi^2 \Omega^2 + \frac{\gamma_1^2}{|z_1-z_3|^2 d^2} + \frac{\gamma_2^2}{|z_2-z_3|^2 d^2} + \frac{\gamma_3^2}{|z_1-z_3|^2 |z_2-z_3|^2}.
\end{gathered} 
\]
Note that these are not yet normalized to collapse to the origin.  Geometrically, this family is found by placing $z_1$ and $z_2$ on the real axis at a distance $d$ from each other and so that the center of their vorticity is at the origin.  The third vortex $z_3$ will then lie on a circle centered at the origin whose radius is determined by the circulations and $d$.  The variable $\theta$  parameterizes this circle.  There are four ``forbidden'' values of $\theta$: two that correspond to a purely rotating configuration ($\kappa = \pm\infty$), and two for which the system for $\cOmega$ has no solutions.  In particular, one can never have an equilateral triangle, which continues the theme from the previous section that collapsing configurations do not display reflection symmetries.

As a concrete example, letting  $\gamma_1=1, \gamma_2=2, d=3$, and $\theta = \pi/2$, we find that
\begin{equation}
\label{point vortex trio}
  \begin{gathered}
    z_1 = -2, \quad z_2 = 1, \quad \gamma_3 = -\frac{2}{3}, \quad z_3 = \sqrt{7} i, \quad \zc = -\frac{2}{\sqrt{7}} i, \\
    \Omega = \frac{35}{264 \pi}, \qquad \frac{1}{\kappa} = \frac{\sqrt{7}}{132 \pi}.
  \end{gathered}
\end{equation}
Shifting each center $z_k \mapsto z_k - \zc$ translates the collapse point to the origin, giving us a solution $\Lambda_0$ to \eqref{vortex V definition}.  Using a computer algebra system, one can verify that $D_\Lambda \mathcal{V}(\Lambda_0)$ does indeed have full rank.  There are several possible decompositions of the parameter space.  For instance, taking either
\[
	\lambda \colonequals (z_1, z_2,\gamma_1,\kappa) \qquad \textrm{or} \qquad \lambda \colonequals (z_3, \gamma_1, \gamma_2,\cOmega),
\]
we find that $D_\lambda \mathcal{V}(\Lambda_0)$ is again invertible.
\end{example}

\begin{example}[Vortex quadruples] \label{collapsing quadruple example}
Consider now the case of a quartet ($M=4$) of self-similarly collapsing point vortices.  Solutions of this type are abundant, though they must usually be determined numerically. In~\cite{kallyadan2022selfsimilar}, the authors obtain the following remarkably explicit example $\Lambda_0$:
\begin{equation}
\label{point vortex quartet}
  \begin{gathered}
    z_1 \colonequals \frac{\sqrt{3}}{2} - \frac{i}{2}, \qquad
    z_2 \colonequals 1+ \frac{\sqrt{3}}{2} -\frac{i}{2}, \qquad 
    z_3 \colonequals -\frac{\sqrt{3}}{2} -1 - \frac{i}{2}, \qquad 
    z_4 \colonequals -\frac{\sqrt{3}}{2} + \frac{i}{2}, \\
    \gamma_1,~ \gamma_4 \colonequals ( 2 \sqrt{3}+5) \pi, \qquad
    \gamma_2,~ \gamma_3 \colonequals (\sqrt{3} -4) \pi, \qquad
    \cOmega \colonequals 2 \sqrt{3} - \frac{3}{4} - \frac{i}{2}.
  \end{gathered}
\end{equation}
Observe that the centers sit at the corners of a parallelogram; we have already normalized them so that $\zc$ is at the origin.  Through computer algebra, the Jacobian $D_\Lambda \mathcal{V}(\Lambda_0)$ can be shown to have full rank, and hence $\Lambda_0$ is non-degenerate.  As in the previous example, there are multiple ways to decompose the parameter space.  For instance, taking
\[
	\lambda \colonequals (z_1, z_2, z_3, \gamma_1, \Omega),
\] 
we compute that $\det{D_\lambda \mathcal{V}(\lambda_0)} \approx 1.397$. 
\end{example}

Larger configurations of collapsing point vortices are also known. Explicit examples for $M=5$ were provided by Novikov and Sedov~\cite{novikov1979vortex}.   As $M$ grows larger, it is typically necessary to use numerical methods to determine the roots of \eqref{vortex V definition}. Kudela~\cite{kudela2104collapse}, for example, finds a configuration consisting of $24$ self-similarly collapsing point vortices.  We direct the reader to the recent survey~\cite{kallyadan2022selfsimilar} for further discussion.

\subsection{Abstract operator equation}

Assume now that $\fullparam_0$ represents parameters describing a configuration of $M \geq 3$ self-similarly collapsing point vortices that is non-degenerate in the sense of Definition~\ref{definition non-degeneracy}.  We may then write $\fullparam_0 = (\param_0,\param_0^\prime)$ so that $D_\lambda \mathcal{V}(\Lambda_0)$ is invertible.  For the desingularization result, the parameters gathered together in $\param^\prime = \param_0^\prime$ will remain constant, whereas $\param$ varies along the continuum of solutions.  

Recall from Section~\ref{complex variable formulation section} that the self-similar hollow vortex system is described through the prescription of a conformal mapping $f$ and complex potential $w$.  For the conformal mapping, we impose the ansatz
\begin{equation}
  \label{f ansatz}
  f = \id + \rho^2 \mathcal{Z}^\rho \mu,
\end{equation}
where the real-valued densities $\mu = (\mu_1, \ldots, \mu_M) \in \mathring{C}^{\ell+\alpha}(\mathbb{T})^M$ become the new unknowns.   Here, $\mathcal{Z}$ is the layer-potential operator defined in~\eqref{definition Z operator}. The factor of $\rho^2$ anticipates the scaling of the governing equations as $\rho \to 0$.   Likewise, we construct $w$ as a perturbation of the velocity potential $w^0$ for the point vortex problem in the plane. Letting
\begin{equation}\label{w0 formula}
  w^0 = w^0(\fullparam)(\zeta) \colonequals  \sum_{k=1}^M \frac{\gamma_k}{2\pi i} \log{\left(\zeta - \zeta_k \right)},
\end{equation}
set 
\begin{equation}
  \label{w ansatz}
  w = w^0 + \rho \mathcal{Z}^\rho \nu 
\end{equation}
for real densities $\nu = (\nu_1, \ldots, \nu_M) \in \mathring{C}^{\ell+\alpha}(\mathbb{T})^M$ to be determined.  Thus the collapsing point vortex configuration corresponds to $(\mu,\nu, \rho) = (0,0,0)$.  

We next rewrite the conformal formulation of the collapsing hollow vortex system~\eqref{conformal governing equations} as a nonlocal problem for these new unknowns.  Unlike the analysis of Section~\ref{m-fold section} where the vortices had $O(1)$ area, special care is needed to ensure that the resulting abstract operators are real-analytic, recover the point vortex problem at $\rho = 0$ with non-degenerate linearization.  Essentially, this amounts to justifying that the Helmholtz--Kirchhoff model accurately describes the limiting behavior the full hollow vortex system as the vortex cores shrink to points.

First consider the kinematic condition~\eqref{kinematic condition}.  Using the layer-potential representations for $f$ and $w$, and supposing that $\rho \neq 0$, the kinematic condition~\eqref{kinematic condition}  is 
\begin{equation}
  \label{nonlocal kinematic undexpanded} 
  0 = \realpart{\left(  \tau \left( w_\zeta^0(\rho \tau + \zeta_k) + \mathcal{Z}_k^\rho\nu^\prime  +  i \cOmega (1+ \rho \mathcal{Z}_k^\rho\mu^\prime) \overline{(\zeta_k + \rho \tau + \rho^2 \mathcal{Z}_k^\rho\mu)}   \right) \right)}
\end{equation}
for all $\tau \in \mathbb{T}$ and $k = 1, \ldots, M$.  Expanding $w^0$ near $\zeta_k$ using \eqref{w0 formula}, we then find that
\begin{equation}
 \label{abstract kinematic condition}
  \begin{aligned}
    0 & = \realpart \Big( \tau \big( \mathcal{Z}_k^\rho \nu^\prime + \Vrho_k^\rho + i \cOmega \overline{( \zeta_k + \rho \tau)}  \rho \mathcal{Z}_k^\rho \mu^\prime + i\cOmega \rho^2 \overline{\mathcal{Z}_k^\rho \mu}  + i \cOmega \rho^3 \mathcal{Z}_k^\rho [\mu^\prime] \overline{\mathcal{Z}_k^\rho\mu} \big) \Big),
  \end{aligned}
\end{equation}
where
\begin{equation}
 \label{definition Lambda}
  \begin{aligned}
    \Vrho_k^\rho = \Vrho_k^\rho(\lambda)(\tau) & \colonequals  \mathcal{V}_k( \fullparam) + \sum_{m=1}^\infty \sum_{j \neq k} {(-1)^m} \frac{\gamma_j}{2 \pi i} \frac{\rho^m \tau^m}{(\zeta_k - \zeta_j)^{m+1}}.
  \end{aligned}
\end{equation}
Observe that $(\rho,\lambda) \mapsto \Vrho_k^\rho(\lambda)$ is real analytic $\mathcal{U} \to C^{\ell+\alpha}(\mathbb{T})$ for any $\ell \geq 0$, where $\mathcal{U}$ is the open neighborhood of $(0,\lambda_0)$  given by 
\begin{equation}
  \label{definition of U neighborhood}
  \mathcal{U} \colonequals  \Big\{ (\rho,\lambda) \in \mathbb{R} \times \mathcal{P} : \min_{j\neq k} |\zeta_j - \zeta_k| > 2|\rho| \Big\}. 
\end{equation}
Note also that the right-hand side of \eqref{abstract kinematic condition} necessarily lies in $\mathring{C}^{\ell-1+\alpha}$ when $\mu, \nu \in \mathring{C}^{\ell+\alpha}$.

Consider now the dynamic condition~\eqref{dynamic condition} on the $k$-the vortex boundary.  As we assume $\gamma_k \neq 0$, the trace of $|w_\zeta|^2$ (and hence $q_k^2$) should diverge like $O(1/\rho^2)$ as $\rho \to 0$ due to \eqref{w0 formula}.    With that in mind, we multiply the left-hand side of \eqref{dynamic condition} by $\rho^2$, evaluate at $\zeta = \rho \tau + \zeta_k$, and then expand to find
\begin{equation}
 \label{definition Bernoulli operator}
  \begin{aligned}
    & \rho^2 \left| \frac{ w_\zeta}{f_\zeta} + i \cOmega \overline{f} \right|^2  \\
    & \quad =  \frac{\left|\dfrac{\gamma_k}{2\pi i \tau} + \rho \left( \mathcal{Z}_k^\rho \nu^\prime + \Vrho_k^\rho \right) + \rho^2     i \cOmega \overline{( \zeta_k + \rho \tau)}  \mathcal{Z}_k^\rho \mu^\prime + \rho^3  i \cOmega \left( \overline{\mathcal{Z}_k^\rho \mu}  + \rho \mathcal{Z}_k^\rho [\mu^\prime] \overline{\mathcal{Z}_k^\rho\mu} \right) \right|^2}{\left|1+ \rho \mathcal{Z}_k^\rho \mu^\prime\right|^2}   \\
    & \quad \equalscolon  \frac{\gamma_k^2}{4\pi^2} + \rho \mathcal{B}_k^\rho + \rho^2 |\cOmega|^2 \left|\zeta_k + \rho \tau + \rho^2 \mathcal{Z}_k^\rho \mu \right|^2 
  \end{aligned}
\end{equation}
This defines an explicit operator $\mathcal{B}_k^\rho = \mathcal{B}_k^\rho(\mu, \nu, \fullparam)$ that is real analytic in a neighborhood of $(\mu,\nu,\rho) = (0,0,0)$.  The dynamic condition can be formulated in terms of the densities via
the requirement
\begin{equation}
  \label{reformulated bernoulli condition} 
  \mathcal{B}_k^\rho(\mu, \nu, \fullparam) = Q_k
  \qquad \text{for $k=1,\ldots,M$}
  ,
\end{equation}
where $Q = (Q_1,\ldots, Q_M) \in \mathbb{R}^M$ are unknown constants.  Notice that in defining $\mathcal{B}_k$, we have incorporated the $|\cOmega|^2 |f|^2$ term in \eqref{dynamic condition}, which we recall accounts for the effects of the moving coordinates on the pressure.  It is easily seen, then, that~\eqref{reformulated bernoulli condition} is equivalent to \eqref{dynamic condition} for $\rho > 0$, and well-defined even for $|\rho| \ll 1$.  At leading-order, moreover,
\begin{equation}
  \label{linear Bernoulli operator}
  \begin{aligned}
  \mathcal{B}_k^\rho(\mu,\nu,\fullparam) & =  
     \frac{\gamma_k^2}{2  \pi^2} \realpart{\left( \frac{2\pi i \tau}{\gamma_k} \left( \mathcal{V}_k(\fullparam) + \mathcal{C}\nu^\prime \right) - \mathcal{C} \mu^\prime\right)} + O(\rho) \qquad \textrm{in } C^{\ell-1+\alpha}(\mathbb{T}). 
  \end{aligned}
\end{equation}
for any $\mu, \nu \in C^{\ell+\alpha}(\mathbb{T})$ and $\ell \geq 1$.  Recall that $\mathcal{C}$ is the Cauchy-type integral operator~\eqref{definition C operator}.

Because $(\mu_k, \nu_k, Q_k)$ describes the behavior near the $k$-th vortex center, it will be convenient to introduce the space
\[  
	(\mu, \nu, Q) \in \Wspace \colonequals   \mathring{C}^{\ell+\alpha}(\mathbb{T})^M \times  \mathring{C}^{\ell+\alpha}(\mathbb{T})^M \times \mathbb{R}^M.
\]
Here $\ell \geq 1$ and $\alpha \in (0,1)$ are fixed but arbitrary. The hollow vortex problem \eqref{conformal governing equations} can then be written as the abstract operator equation
\[
	\F( u;\, \rho) = 0,
\]
where we are now considering $\F = (\A, \B) \colon \mathscr{O} \subset \Xspace \times \mathbb{R} \to \Yspace$ to be the real-analytic mapping between the spaces
\begin{align*}
	\Xspace   \colonequals  \Wspace \times \Pspace, \qquad
	\Yspace  \colonequals   \mathring{C}^{\ell-1+\alpha}(\mathbb{T})^M \times C^{\ell-1+\alpha}(\mathbb{T})^M,
\end{align*}
whose components $\A_k$ and $\B_k$ enforce the kinematic and Bernoulli conditions on the $k$-th vortex boundary, respectively.  Explicitly, they are     
\begin{equation}
 \label{abstract operator}
  \begin{aligned}
    \mathscr{A}_k(u, \rho) 
    & \colonequals  \realpart{\left( \tau \left( \mathcal{Z}_k^\rho \nu^\prime + \Vrho_k^\rho + ( i \cOmega \overline{( \zeta_k + \rho \tau)} -c) \rho \mathcal{Z}_k^\rho \mu^\prime + i\cOmega \rho^2 \overline{\mathcal{Z}_k^\rho \mu}  + i \cOmega \rho^3 \mathcal{Z}_k^\rho [\mu^\prime] \overline{\mathcal{Z}_k^\rho\mu} \right) \right)}   \\
    \mathscr{B}_k(u, \rho)	& \colonequals  \mathcal{B}_k^\rho(\mu,\nu,\param) - Q_k,
  \end{aligned}
\end{equation}
where recall that the function $\Vrho^\rho$ is given by \eqref{definition Lambda}.  
We likewise redefine the domain of $\F$ to be the open set
\begin{equation}
  \label{definition O delta}
  \mathscr{O} \colonequals  \left\{ (\mu,\nu,Q,\param,\rho) \in \Wspace \times \mathcal{U} : \inf_k \inf_{\mathbb{T}}{ |1+ \rho \mathcal{Z}_k^\rho \mu_k^\prime |} > 0 \right\},
\end{equation}
with $\mathcal{U}$ as in \eqref{definition of U neighborhood}.  Membership in $\mathscr{O}$ ensures that the corresponding mapping $f$ constructed via \eqref{f ansatz} has nonvanishing derivative on $\partial \confD_\rho$. The trivial solution corresponding to the point vortex configuration now takes the form
\[
	(\mu,\nu,Q,\param,\rho)  = (0,0,0,\param_0,0) \equalscolon  (u^0,0) \in \mathscr{O}.
\]
In order for $(u,\rho) \in \F^{-1}(0)$ to represent a physical solution to the problem, it is additionally necessary that $f$ is univalent on $\overline{\confD_\rho}$.  This will follow for the small hollow vortex solutions by construction, as $f$ will be near identity.

\subsection{Proof of Theorem~\ref{intro configurations theorem}} 
 We now state and prove a rigorous version of Theorem~\ref{intro configurations theorem}. For convenience, this will be done in the conformal formulation~\eqref{conformal governing equations} and using the unknowns $(\mu,\nu,Q, \rho)$.  With that in mind let $\fullparam_0 = (\param_0, \param_0^\prime)$ be a non-degenerate collapsing $M$ point vortex configuration.
 Let
 \[
	S_k \colonequals  -\frac{1}{2} \sum_{j \neq k} \frac{\gamma_j^0}{2 \pi i} \frac{1}{(\zeta_j^0 - \zeta_k^0)^2}, \qquad k = 1, \ldots, M.
\]
As in \cite{llewellyn2012structure}, one can understand $S_k$ as the effective straining velocity field experienced by the $k$-th hollow vortex due to the influence of the other vortices.  

The main result is then the following.

\begin{theorem}[Vortex desingularization] \label{imploding hollow vortex configuration theorem}
Let $\fullparam_0 = (\param_0, \param_0^\prime)$ be a non-degenerate steady vortex configuration.  There exists $\rho_1 > 0$ and a curve $\km$ of solutions to the imploding vortex problem~\eqref{conformal governing equations} admitting the real-analytic parameterization 
\[
	\km = \left\{  (u^\rho,\rho) = (\mu^\rho, \nu^\rho, Q^\rho, \param^\rho, \rho)  \in \mathscr{O} : \quad |\rho| < \rho_1 \right\} \subset \F^{-1}(0).
\]
The curve $\km_\loc$ bifurcates from the point vortex configuration $\fullparam_0$ in that $u^0 = (0, \param_0)$, and to leading order we have
	\begin{equation}
	\label{leading order hollow vortex}
    \begin{aligned}
      \mu_k^\rho(\tau) & =   \frac{16\pi}{\gamma_k} \rho \realpart{ \left( i S_k \tau \right)}  + O(\rho^2) &  \textrm{in } \mathring{C}^{\ell+\alpha}(\mathbb{T}) \\
      \nu_k^\rho(\tau) & = -2 \rho   \realpart {\left( S_k \tau^2 \right)}  +O(\rho^2) &  \textrm{in } \mathring{C}^{\ell+\alpha}(\mathbb{T})\\
      Q^\rho & =  \rho |\cOmega^0|^2 |\zeta_k^0|  + O(\rho^3) \\ 
      \param^\rho & = \param_0 + O(\rho^2).
    \end{aligned}
\end{equation} 
\end{theorem}
\begin{remark}
Comparing the asymptotic formulas~\eqref{leading order hollow vortex} to the corresponding ones obtained for rotating, translating, or stationary configurations in~\cite[Theorem 6.1]{chen2023desingularization}, we see that the only change is that $Q^\rho$ here is $O(\rho)$ rather than $O(\rho^2)$. 
  Recall that here have included the inhomogeneous $|\cOmega|^2 |f|^2$ term in the dynamic condition~\eqref{dynamic condition} so that the pressure is (spatially) constant in the lab frame, while this term was absent in \cite{chen2023desingularization}.
\end{remark}

\begin{proof}
The Fréchet derivative of $\F$ at the fixed collapsing point vortex configuration is given by
\[
	 D_u \F(u^0,0) \dot u = 
				\begingroup
				\renewcommand\arraystretch{1.5}
				\begin{pmatrix} 
					 0  & \A_\nu^0 & 0 &  \A_\param^0 \\
					 \B_\mu^0 & \B_\nu^0 & \B_Q^0 &  \B_\param^0  \\
				\end{pmatrix} \endgroup
				\begin{pmatrix} \dot \mu \\ \dot \nu \\ \dot Q \\ \dot \param \end{pmatrix},
\]
where we are abbreviating $\A_\nu^0 \colonequals  D_\nu \A(u^0,0)$ and so on.  It is likewise quickly verified that $\A_\nu^0$ and $\B_w^0$ have have block diagonal form 
\begin{equation}
\label{DA DB formulas}
\begin{aligned}
 	\A_{k\nu}^0 \dot\nu & = \realpart{\left( \tau \mathcal{C} \dot\nu_k^\prime \right)}, \\ 
	\A_{k\rho}^0 & = \realpart{ \left(  \tau  \Vrho_{k\rho}^0\right)} = 2\realpart{(S_k \tau^2)}, \\
  	\B_{k(\mu,\nu,Q)}^0\begin{pmatrix} \dot \mu \\  \dot\nu \\ \dot Q\end{pmatrix} & = \frac{(\gamma_k^0)^2}{2 \pi^2} \realpart{ \left(  \frac{2\pi i}{\gamma_k^0} \tau \mathcal{C} \dot\nu_k^\prime - \mathcal{C}\dot\mu_k^\prime \right)} -\dot Q_k,
\end{aligned}
\qquad \textrm{for } k = 1, \ldots, M.
\end{equation}
These agree exactly with the linearized steady hollow vortex system computed in~\cite[Section 6]{chen2023desingularization}; the collapsing case differs only in the form of $\mathcal{V}$ and hence $\A_{k\lambda}^0$ and $\B_{k\lambda}^0$.  Thanks to this similarity, arguing directly as in the proof of~\cite[Theorem 6.1]{chen2023desingularization}, we conclude that 
 $D_u \F(u^0,0) \colon \Xspace \to \Yspace$ is Fredholm with
\begin{equation}
\label{Fu dimension counts}
  \begin{aligned}
    \dim\kernel{D_u \F(u^0,0)} & = \dim\kernel{D_{\param} \mathcal{V}(\fullparam_0)} \\ 
    \cod\range{D_u\F (u^0, 0)} &= \cod\range{D_{\param} \mathcal{V}(\fullparam_0)}. 
  \end{aligned}
\end{equation}  
By hypothesis, $\Lambda_0$ is a non-degenerate point vortex configuration, and hence $D_\param \mathcal{V}(\fullparam_0)$ is an isomorphism.  Then, \eqref{Fu dimension counts} ensures the same is true of $D_u\F(u^0,0)$ as a mapping $\Xspace \to \Yspace$, so that the existence of the curve $\km$ follows from the (real-anayltic) implicit function theorem.  

It remains only to verify that the solutions along it have leading-order form stated in~\eqref{leading order hollow vortex}.  Writing $u^\rho \equalscolon  \rho \dot u + O(\rho^2)$, we see that
\[
	\A_u^0 \dot u = -\A_\rho^0, \qquad \B_u^0 \dot u = -\B_\rho^0.
\] 
The formulas for $\A_w^0$ and $\B_w^0$ are already given in \eqref{DA DB formulas}.  For the remaining derivatives, we compute that
\begin{align*}
	 \A_{k\lambda}^0 &= \realpart{ \left( \tau  \mathcal{V}_{k\param} \right)} \\
	 \A_{k\rho}^0 & = \realpart{ \left(  \tau  \Vrho_{k\rho}^0\right)} = -\realpart{\left( \frac{1}{2\pi i} \tau^2 \sum_{j \neq k} \frac{\gamma_j^{0}}{(\zeta_j^0 - \zeta_k^0)^2}\right) } = 2\realpart{(S_k \tau^2)}.
\end{align*}
Then, from
  \begin{equation}
    \label{asymptotic kinematic}
    \A_{k\nu}^0 \dot\nu + \A_{k \param}^0 \dot \param = - \A_{k\rho}^0,
  \end{equation}
we see that
\[
	0 = \proj_1 \A_{k\param}^0 \dot\param = \frac{\tau}{2} \Vrho_{k\lambda} \dot\lambda+  \frac{1}{2\tau} \overline{\Vrho_{k\lambda} \dot\lambda}.
\]
But, recall that by construction $\Vrho_\lambda$ has a trivial kernel, and thus $\dot\lambda = 0$. Applying $\proj_{>2}$ to \eqref{asymptotic kinematic} yields 
\[
	\proj_{> 2}  \A_{k\nu}^0 \dot\nu = 0,
\]
and hence $\proj_{>2} \dot\nu = 0$.  Moreover, inverting $\realpart{(\tau \mathcal{C} \partial_\tau \placeholder)}$ gives
\begin{align*}
	 \dot\nu_k &=   -\left( S_k \tau^2 + \overline{S_k} \frac{1}{\tau^2} \right) 
	 =-2 \realpart{\left(  S_k \tau^2 \right)}.
\end{align*}

Turning to the dynamic condition, we see that
  \begin{equation}
    \label{asymptotics bernoulli}
    -\frac{(\gamma_k^0)^2}{2 \pi^2} \realpart{ \mathcal{C} \dot\mu_k^\prime} + \frac{\gamma_k^0}{\pi} \realpart{(i \tau \mathcal{C} \dot\nu_k^\prime)} - \dot Q_k + \B_{k\param}^0 \dot\param = -\B_{k\rho}^0.
  \end{equation}
In view of~\eqref{definition Bernoulli operator}, expanding the operator $\mathcal{B}_k$ gives
\begin{align*}
	\rho \mathcal{B}_k^\rho(0,0,\lambda) & = \left| \frac{\gamma_k}{2\pi i \tau} + \rho \Vrho_k^\rho\right|^2 - \frac{\gamma_k^2}{4 \pi^2} - \rho^2 |\cOmega| |\zeta_k + \rho \tau| \\
		& = \rho^2 |\Vrho_k^\rho|^2 + 2 \rho \realpart{\left( \overline{\frac{\gamma_k}{2\pi i \tau}} \Vrho_k^\rho \right)} - \rho^2 |\cOmega^0|^2 |\zeta_k^0| +O(\rho^3),
\end{align*}
and thus
\begin{align*}
	\B_{k\rho}^0 &=  \frac{\gamma_k^0}{\pi } \realpart{\left(  i\tau \Vrho_{k\rho}^0 \right)}  - |\cOmega^0|^2 |\zeta_k^0| = \frac{2\gamma_k^0}{\pi} \realpart{ (i S_k \tau^2  )} - |\cOmega^0|^2 |\zeta_k^0|.
\end{align*}
In particular, applying the projection $\proj_0$ to \eqref{asymptotics bernoulli} reveals that $\dot Q_k = |\cOmega^0|^2 |\zeta_k^0|$.  Using the previous expression for $\dot\nu$, we then find that
\begin{align*}
	\frac{\gamma_k^0}{\pi} \realpart{(i \tau \mathcal{C} \dot\nu_k^\prime)} &=  \frac{2\gamma_k^0}{\pi} \realpart{\left( i S_k \tau^2 \right)}.
\end{align*}
Applying $\proj_2$ and $1-\proj_2$ to \eqref{asymptotics bernoulli}, we therefore obtain
\[
	\left\{
	\begin{aligned}
		(1-\proj_2) \realpart{ \mathcal{C} \dot\mu_k^\prime} & =  0 \\
		\proj_2 \realpart{ \mathcal{C} \dot\mu_k^\prime} & = \proj_2 \left(  \frac{2\pi^2}{(\gamma_k^0)^2} \B_{k\rho}^0 + \frac{2 \pi}{\gamma_k^0} \realpart{\left( i  \tau \mathcal{C}\dot\nu_k^\prime\right)} \right) = \frac{ 8 \pi}{\gamma_k^0} \realpart{ \left(  i S_k \tau^2 \right)}.
	\end{aligned}
	\right.
\]
Inverting once more the Fourier multiplier $\realpart{\mathcal{C} \partial_\tau}$, we find at last that
\begin{align*}
	\dot\mu_k & = \frac{16 \pi}{\gamma_k^0} \realpart{ \left(  i S_k \tau \right)},
\end{align*}
which completes the proof.
\end{proof}

As an immediate consequence of Theorem~\ref{imploding hollow vortex configuration theorem} and the computed Jacobians in Examples~\ref{collapsing triple example} and \ref{collapsing quadruple example}, we then have the following.

\begin{corollary}[Imploding trios and quartets] 
\label{triple quadruple corollary}
Fix $\ell \geq 1$ and $\alpha \in (0,\alpha)$.
\begin{enumerate}[label=\rm(\alph*\rm)] 
	\item \textup{(Trio)} There exists a curve $\km_3$ of imploding hollow vortices trios bifurcating from the collapsing point vortex trio~\eqref{point vortex trio}.
	\item \textup{(Quartet)}  There exists a curve $\km_4$ of imploding hollow vortices quartets bifurcating from the collapsing point vortex quartet~\eqref{point vortex quartet}. 
\end{enumerate}
\end{corollary}

\appendix
\section{Comparison of asymptotics for rotating vortices}
\label{asymptotics appendix}

In this Appendix, we record the higher-order terms in the expansion of the rotating hollow vortex solutions along the curves $\cm_{\rot}^{m,\pm}$ constructed in Theorem~\ref{m-fold rotating theorem}. As these solutions are obtained via the analytic Crandall--Rabinowitz theorem, the asymptotics can be found recursively through an elementary but lengthy calculation. 
Recall that there are two branches of the dispersion relation,
\begin{align}
  \label{dispersion again}
  \Omega_{0,\pm} = \frac {\gamma}{2 \pi} 
  \frac{\sqrt m \pm 1}{\sqrt m},
\end{align}
and $\cm_{\rot}^{m,\pm}$ bifurcates from the trivial rotating circular solution at $\Omega = \Omega_{0,\pm}$. 
Third-order asymptotic expansions for these branches of
solutions are given by
\begin{equation}
\label{higher-order asymptotics}
  \begin{aligned}
    f_\pm^\varepsilon(\zeta) &= \zeta
    + \varepsilon \zeta^{1-m}
    - \varepsilon^2 (\sqrt m \pm 1)^2 \zeta^{1-2m}
    + \tfrac 32 \varepsilon^3 (\sqrt m \pm 1)^4 \zeta^{1-3m}
    + O(\varepsilon^4), \\
    \Big(\frac \gamma{2\pi i}\Big)^{-1} w_\pm^\varepsilon(\zeta) &= \log \zeta
    + \varepsilon \frac{\sqrt m \pm 1}{\sqrt m} \zeta^{-m}
    - \varepsilon^2 \frac{(\sqrt m \pm 1)^3}{\sqrt m} \zeta^{-2m}
    \\ &\qquad\qquad 
    + \varepsilon^3 \left( 
      \frac{3(\sqrt m \pm 1)^5}{2\sqrt m} \zeta^{-3m} 
      \pm \frac{(\sqrt m \pm 1)^4}{2\sqrt m} \zeta^{-m} 
      \right)
      + O(\varepsilon^4), \\
    \Big(\frac \gamma{2\pi}\Big)^{-1} \Omega_\pm^\varepsilon &= 
    \frac{\sqrt m \pm 1}{\sqrt m} 
    \pm \varepsilon^2 \frac{(\sqrt m \pm 1)^3(\sqrt m \pm 3)}{2\sqrt
    m} + O(\varepsilon^4),\\
    \Big(\frac \gamma{2\pi}\Big)^{-2} q_\pm^\varepsilon &= 
    - \frac{\sqrt m \pm 2}{2\sqrt m} 
    \mp \varepsilon^2 \frac{(\sqrt m \pm 1)^2
    (m \pm 3\sqrt m + 3)}{2\sqrt m} 
    + O(\varepsilon^4).
  \end{aligned}
\end{equation}
Here $f_\pm^\varepsilon$ is the conformal mapping corresponding to $\mu_\pm^\varepsilon$ while $w_\pm^\varepsilon$ is the holomorphic function determined by $\nu_\pm^\varepsilon$ via the ansatz~\eqref{m-fold trace ansatz}.

Performing a similar analysis for the H-state problem considered by Crowdy, Nelson, and Krishnamurthy~\cite{crowdy2021hstates}, which recall amounts to omitting the
$\Omega^2 |f|^2$ term in the dynamic boundary condition~\eqref{dynamic condition}, we
instead find
\begin{align*}
  f_\hstate(\zeta) &= \zeta
    + \varepsilon \zeta^{1-m}
    - \varepsilon^2 \frac{(m-1)^2}{4m} \zeta^{1-2m}
    + \varepsilon^3 \frac{(m-1)^4}{16m^2} \zeta^{1-3m}
    + O(\varepsilon^4), \\
  \Big(\frac \gamma{2\pi i}\Big)^{-1} w_\hstate(\zeta) &= \log \zeta
    + \varepsilon \frac{m-1}{m+1} \zeta^{-m}
    - \varepsilon^2 \frac{(m-1)^3}{4m(m+1)} \zeta^{-2m}
    + \varepsilon^3 \frac{(m-1)^5}{16m^2(m+1)} \zeta^{-3m} 
    + O(\varepsilon^4), \\
  \Big(\frac \gamma{2\pi}\Big)^{-1} \Omega_\hstate &= 
    \frac{m-1}{m+1} + \varepsilon^2 \frac{(m-1)^2}{4m(m+1)}
    + O(\varepsilon^4),\\
  \Big(\frac \gamma{2\pi}\Big)^{-2} q_\hstate &= 
    \frac 2{(m+1)^2}
    + \varepsilon^2 \frac{(m-1)^2}{2m(m+1)}
  + O(\varepsilon^4).
\end{align*}
This expansion for the conformal mapping agrees with the exact
formula
\begin{align*}
  f_\hstate(\zeta) = \zeta + \frac{\varepsilon \zeta}{\zeta^m - \tfrac 14
  \varepsilon m^{-1}(m-1)^2}
\end{align*}
which can be obtained by suitably transforming the solution in~\cite{crowdy2021hstates}.

Finally, it is interesting to compare the solutions along $\cm_{\rot}^{m,\pm}$ to the well-known \emph{V-states}, which are $m$-fold symmetric rotating vortex patches.  Su~\cite{su1979motion} provides a higher-order expansions for V-states in terms of  a polar representation $r=R_\vstate(\theta)$ of the patch boundary.
In dimensionless units where the radius of the unperturbed patch is
$1$, and letting $\omega$ denote the constant vorticity in the patch, these are
\begin{align*}
  R_\vstate(\theta) &= 1+ \varepsilon \cos (m\theta) + \varepsilon^2
  \frac{2m-1}4 \cos(2m\theta) + \varepsilon^3 \frac{(3m-1)(m-1)}8 \cos(2m\theta) + O(\varepsilon^4), \\
  \omega^{-1} \Omega_\vstate &= \frac{m-1}{2m} - \varepsilon^2 \frac{m-1}4 +
  O(\varepsilon^4).
\end{align*}

One important way in which our solution families differ from the H-states and V-states are in terms of the streamline pattern. Recall that 
critical layers are curves along which the angular component of
the relative fluid velocity 
\begin{align*}
  -\frac{\boldsymbol \xi^\perp}{|\boldsymbol \xi|} \cdot \mathbf U
  = 
  \Omega |f| + \imagpart \frac{fw_\zeta}{|f| f_\zeta} 
\end{align*}
vanishes. These will be present in the
perturbative solutions only if they are present at the bifurcation
point. Using the dispersion relation \eqref{dispersion again}, we see that the bifurcation points for the $\cm_{\rot}^{m,+}$ branches have no critical layers, while those for the $\cm_{\rot}^{m,-}$ branches have critical layer at the radius
\begin{align*}
  |\boldsymbol \xi|^2 = \frac{\sqrt m}{\sqrt m-1}. 
\end{align*}
For comparison, the perturbative H-states have critical layers at the
radius
\begin{align*}
  |\boldsymbol \xi|^2 = \frac{m+1}{m-1},
\end{align*}
while the rotating vortex patches have critical layers at the radius
\begin{align*}
  |\boldsymbol \xi|^2 = \frac m{m-1}.
\end{align*}
In all cases, these critical layers approach the unperturbed interface
$|z|=1$ as the symmetry class $m$ becomes large. See~\cite{park2022quantitative} for more on this limit in the V-state case.

Another distinctive feature of the $\cm_{\rot}^{m,-}$ branches is that the nature of the local bifurcation changes depending on $m$. From~\eqref{higher-order asymptotics}, we see that the bifurcation diagram for both $\Omega_+^\varepsilon$ and $\Omega_-^\varepsilon$ is a pitchfork. For all 
$\cm_{\rot}^{m,+}$ branches with $m \geq 2$, and for all $\cm_{\rot}^{m,-}$ branches with $2 \leq m \leq 8$, that pitchfork is supercritical. However, for $m \geq 10$, the pitchforks for $\cm_{\rot}^{m,-}$ become subcritical:
\begin{align*}
  \Omega_-^\varepsilon 
  \begin{cases}
    > \Omega_{0,-} & m \le 8,\\
    < \Omega_{0,-} & m \ge 10
  \end{cases} 
  	\qquad \textrm{for all } 0 < \varepsilon \ll 1.
\end{align*}
The case $m=9$ is special, as we can see from~\eqref{higher-order asymptotics} that $\Omega_-^\varepsilon - \Omega_{0,-} = O(\varepsilon^4)$, meaning an even higher-order expansion must be performed to discern the direction. This is in sharp contrast to both H-states, which are always supercritical pitchforks, and V-states, which are always subcritical pitchforks.

\section*{Acknowledgments} 

The research of RMC is supported in part by the NSF through DMS-2205910. The research of SW is supported in part by the NSF through DMS-2306243, and the Simons Foundation through award 960210.

\section*{Data availability} 
There is no data associated to this manuscript.

\section*{Conflict of interest} 

The authors confirm that they have no conflict of interest regarding this work.

\bibliographystyle{siam}
\bibliography{projectdescription}

\end{document}